\documentclass[11pt,reqno,bold]{paper}
\pdfoutput=1
\usepackage{geometry}
\usepackage{amsmath}
\usepackage{amssymb}
\usepackage{amsfonts}

\usepackage[round]{natbib}
\usepackage[algoruled,lined,longend,linesnumbered]{algorithm2e}
\usepackage{enumerate}

\usepackage{amsthm}
\newtheorem{theorem}{Theorem}[section]
\newtheorem{proposition}[theorem]{Proposition}
\newtheorem{lemma}[theorem]{Lemma}

\theoremstyle{remark}
\newtheorem{remark}[theorem]{Remark}

\usepackage{color}
\definecolor{darkblue}{rgb}{0.0,0.0,0.4}
\usepackage[pdfauthor={Tsogtgerel Gantumur},pdftitle={Adaptive BEM with convergence rates},pdfsubject={analysis of adaptive algorithms},colorlinks,citecolor=darkblue,linkcolor=darkblue,urlcolor=darkblue]{hyperref}

\newcommand{\N}{\mathbb{N}}

\newcommand{\R}{\mathbb{R}}
\newcommand{\eps}{\varepsilon}
\newcommand{\vol}{\mathrm{vol}}

\newcommand{\diam}{\mathrm{diam}}
\newcommand{\osc}{\mathrm{osc}}
\newcommand{\len}{|\!|\!|}
\newcommand{\ren}{|\!|\!|}

\newcommand{\exd}{\mathrm{d}}
\newcommand{\supp}{\mathrm{supp}}
\newcommand{\dist}{\mathrm{dist}}
\newcommand{\intr}{\mathrm{int}}
\newcommand{\card}{\mathrm{card}}
\newcommand{\adm}{\mathrm{adm}}
\newcommand{\spn}{\mathrm{span}}
\newcommand{\refine}{\mathrm{refine}}

\newcommand{\shape}{\sigma_{\textrm{s}}}
\newcommand{\grade}{\sigma_{\textrm{g}}}

\newcommand{\supb}{{\Omega}}
\newcommand{\mns}{{\nu}}

\begin{document}

\title{Adaptive boundary element methods with convergence rates}
\author{Tsogtgerel Gantumur}
\institution{McGill University}
\date{\today}

\maketitle


\begin{abstract}
This paper presents adaptive boundary element methods for positive, negative, as well as zero order operator equations, 
together with proofs that they converge at certain rates.
The convergence rates are quasi-optimal in a certain sense under mild assumptions that are analogous to what is typically assumed in 
the theory of adaptive finite element methods.
In particular, no saturation-type assumption is used.
The main ingredients of the proof that constitute new findings are some results on \emph{a posteriori} error estimates for boundary element methods,
and an inverse-type inequality involving boundary integral operators on locally refined finite element spaces.
\end{abstract}

\section{Introduction}

Let $\Gamma$ be a (closed or open) polyhedral surface in $\R^3$.
We consider equations of the form
\begin{equation}
  \label{e:op0}
  Au=f,
\end{equation}
where $f\in H^{-t}(\Gamma)$, and $A:H^{t}(\Gamma)\to H^{-t}(\Gamma)$
is an invertible linear operator.
Strictly speaking, the function spaces should be slightly modified if $\Gamma$ is an open surface or $t>0$,
but for the sake of this introduction we will gloss over this point.
The operators of interest are the \emph{boundary integral operators} that arise
from reformulations of boundary value problems as integral equations
on the domain boundary, cf. \cite{McL00}.
Then the problem \eqref{e:op0} corresponds to a \emph{boundary integral
equation},
and a very popular class of methods for its numerical solution is 
\emph{boundary element methods} (BEM),
which can crudely be described as finite element methods (FEM) applied to boundary
integral equations, cf. \cite{SS11}.
As with finite elements, there is the \emph{adaptive} version of BEM,
whose main feature is to automatically distribute mesh points hopefully in an optimal way
so as to obtain an accurate numerical solution.
Even though those methods perform very well in practice, their
mathematical theory is not in a very satisfactory state, especially
if one compares it with the corresponding theory of adaptive FEM.
In the latter context, the sequence of papers
\citet{Dorf96}, \citet*{MNS02}, \citet*{BDD04}, \cite{Stev07}, and
\citet*{CKNS08}
has laid a fairly satisfactory foundation to mathematical understanding of
adaptive FEM for linear elliptic boundary value
problems.
Specifically, it was established that standard adaptive FEM generate a sequence of solutions whose error decreases
geometrically, and that the number of triangles in the mesh grows with
an optimal rate.
Moreover, rigorous treatments of numerical integration and linear algebra solvers are within reach.

In parallel to the above, a very general theory of adaptive wavelet methods has been developed, cf.
\citet*{CDD01,CDD02,GHS07}. 
Under this framework, one can analyze convergence rates and complexity of fully discrete adaptive wavelet methods for boundary integral equations, cf. \cite*{Stev04,GS06b,DHS07}

However, there has been a significant gap in the current mathematical understanding
of adaptive BEM proper. 
The first steps toward closing the gap have been
taken in \cite{CP09}, where convergence is guaranteed for an adaptive
BEM with a feedback control that occasionally adds uniform refinements,
in \citet*{FLOP10} where geometric error reduction is proven under
a \emph{saturation assumption},
and in \citet*{AFLP10}, where convergence is established under a weak
saturation-type assumption.
To give an idea of what a saturation assumption is,
one form of it requires that if $u_k$ is the numerical solution of \eqref{e:op0} at
the current stage of the algorithm execution, and if $\hat{u}_k$ is
the would-be numerical solution had we replaced the current mesh by
its uniform refinement, then $\|u-u_k\|_{H^t(\Gamma)}\leq \beta \|\hat{u}_k-u_k\|_{H^t(\Gamma)}$,
where $\beta>0$ is a constant.
Such assumptions were common in the finite element literature 
before \cite{Dorf96} and \cite{MNS02} proved geometric
error reduction without relying on a saturation assumption.

In this paper, we prove geometric error reduction for
three kinds of adaptive boundary element methods for positive ($t>0$), negative ($t<0$), as well as zero
order operator equations, without using a saturation-type assumption.
In fact, several types of saturation assumptions follow from our work as a corollary.
Moreover, bounds on the convergence rates are obtained that are in
a certain sense optimal.

Essentially at the same time as this work became available as an arXiv preprint, and independently of this work, preprints
\citet*{FKMP11,FKMP11a} appeared in which the authors prove
the same type of results for an adaptive BEM for a negative
order operator equation, by methods that are not dissimilar but largely complementary to ours.
In a certain sense, the results of this work and of \citet{FKMP11,FKMP11a} together
bring the mathematical understanding of adaptive BEM to the level comparable to that of adaptive FEM.

A large chunk of the techniques developed in the adaptive finite element
theory is not specific to differential equations, 
providing a nice starting point for us and a clue on what ingredients are
missing in the boundary element theory.
In this work, we needed to supply two kinds of ingredients that constitute new findings:
some results on \emph{a posteriori} error estimators for BEM and an
inverse-type inequality involving boundary integral operators and
locally refined meshes.
The issue with the theory of error estimators has been most obvious;
what is usually guaranteed is only one of the two inequalities that are necessary for a
quantity to resemble the error.
There is a very few estimators with both \emph{upper and lower bounds} proven; one can
mention the estimators proposed in \cite{Faer98,Faer00,Faer02}, see
\cite{CF01} for a more thorough discussion. 
To name some of the relatively recent works on this subject,
we have \citet*{CMS01}, \citet*{CMPS04},
and \citet*{NPZ10}, where upper bounds for certain residual-type 
estimators are established,
and \citet*{EFLFP09}, and \citet*{EFGP09}, where upper and lower bounds are proven for a large number of non-residual type estimators, with the upper bounds depending on various forms of saturation assumptions.
Now, even if one had both upper and lower bounds for an error estimator, one still
needs a so-called \emph{local discrete bound} before attempting to apply the
general techniques from the finite element theory. 
Such an estimate has been entirely open for boundary element methods.

In this work, we establish all missing bounds for a number of residual-type error estimators for positive, negative, as well as zero order
boundary integral equations, including the estimators from
\cite{CMS01}, \cite{CMPS04}, and \cite{Faer00,Faer02}.
The recent papers \cite{FKMP11,FKMP11a} also contain similar results regarding the estimator from \cite{CMS01}.
Some of our bounds involve the so-called \emph{oscillation terms},
that give useful estimates on how far the current mesh is from saturation.
Note that analogous terms also arise in the finite element theory.
In order to handle the oscillation terms, which turned out to be not straightforward,
we prove an \emph{inverse-type inequality} involving
boundary integral operators and locally refined meshes.
Our proof of the inverse-type inequality in general requires the underlying surface $\Gamma$ to be
$C^{1,1}$ or smoother, but for open surfaces it allows the boundary of
$\Gamma$ to be Lipschitz.
So in general, polyhedral surfaces are ruled out. 
However, this is very likely an artifact of the proof,
since in \cite{FKMP11,FKMP11a},
the inequality is proven for a model negative order operator on polyhedral surfaces.

This paper is organized as follows.
In the next section, we fix the general setting of the paper, and
recall some basic results that will be used throughout the paper.
Then in the following three sections, namely in \S\ref{s:zero},
\S\ref{s:positive}, and \S\ref{s:negative}, we consider adaptive BEMs
for zero, positive, and
negative order operator equations, respectively.
In each of the three cases, we first study certain residual based \emph{a
posteriori} error estimators, then design an adaptive BEM based on those
estimators, and finally address the question of
convergence rate.
There are some general arguments and remarks that can be applied to
all of the three cases, and so in order not to be too repetitive,
they shall be discussed in \S\ref{s:zero} for the zero order case
in detail, and then in \S\ref{s:positive} and \S\ref{s:negative}, 
we will simply refer to them if needed.
The entire analysis depends on an inverse-type inequality involving
boundary integral operators and locally refined meshes, 
which then is verified in \S\ref{s:inverse} for a general class of boundary
integral operators.
We will conclude the paper by summarizing the results and making a
list of important open problems.

\section{Generalities}
\label{s:general}

In this section, we will set up the necessary vocabulary and collect some basic results that will be used throughout the paper.
Let $\Omega$ be a compact closed, $n$-dimensional, patchwise smooth, globally
$C^{\mns-1,1}$ manifold,
which will be the habitat of all functions and distributions that ever
occur in this paper.
In practice, we typically have $n=1$ or $n=2$,
so in the discussions that follow we often use the language of $n=2$, 
and we shall indicate whenever we lose generality by implicitly restricting to the
case $n=2$.
We will assume that each smooth (closed) patch of $\Omega$ is
diffeomorphic to a polygon,
so one can think of $\Omega$ as the surface of a bounded polyhedron, with
faces and edges now allowed to be smooth surfaces and curves, respectively.

\subsection{Sobolev spaces}
For $s\in[-\mns,\mns]$, we denote by $H^{s}(\Omega)$ the usual Sobolev
space of order $s$ on $\Omega$.
Let $\omega\subseteq\Omega$ be an open subset of $\Omega$ with Lipschitz boundary.
Then we define the following two kinds of Sobolev spaces
\begin{equation}\label{e:sobolev}
H^s(\omega)=\{u|_\omega:u\in H^s(\Omega)\},
\qquad\textrm{and}\qquad
\tilde{H}^s(\omega)=\{u\in H^s(\Omega):\supp\,u\subseteq\overline{\omega}\},
\end{equation}
for $s\in[-\mns,\mns]$.
Obviously $\tilde{H}^s(\Omega)=H^s(\Omega)$, and it is known that
$\tilde{H}^s(\omega)^* = H^{-s}(\omega)$, cf. \cite{McL00}.
Note also that $\tilde{H}^s(\omega)$ is a closed subspace of
$H^s(\omega)$ for $s\geq0$.
The definitions \eqref{e:sobolev} give rise to the canonical norms on $H^s(\omega)$ and $\tilde{H}^s(\omega)$ inherited from the norm on $H^s(\Omega)$,
and at least for $s>0$, these norms are known to be equivalent to certain norms defined either by interpolation, or in terms of moduli of smoothness, with the equivalence constants depending only on the dimension $n$, the order $s$, the Lipschitz constant of $\omega$, and the particulars (i.e., the local coordinate patches and the partition of unity) in the definition of $H^s(\Omega)$.
Let us expand on this a bit.
Since the spaces with $s>1$ are of secondary importance to us, for simplicity here we focus only on the case $s\in[0,1]$, hence by duality on $|s|\leq1$.
With $\|\cdot\|_{\omega}$ and $|\cdot|_{1,\omega}$ denoting the $L^2$-norm and the usual $H^1$-seminorm on $\omega$, respectively, we define $|\cdot|_{H^s(\omega)}$ to be the interpolatory seminorm between $\|\cdot\|_{\omega}$ and $|\cdot|_{1,\omega}$,
for concreteness by using the $K$-functional.
Then $\|\cdot\|_{\tilde{H}^s(\omega)}:=|\cdot|_{H^s(\omega)}$ is a norm on $\tilde{H}^s(\omega)$, which can be made into a norm $\|\cdot\|_{H^s(\omega)}$ on $H^s(\omega)$ by combining it with the $L^2$-norm.
It can be shown that these norms are equivalent to the canonical norms on $\tilde{H}^s(\omega)$ and respectively $H^s(\omega)$,
with equivalence constants depending only on $n$, $s$,  and the particulars of the definition of $H^s(\Omega)$, cf. \cite{McL00}.
Moreover, we have the following useful property that if $\omega_1,\ldots,\omega_k$ are disjoint Lipschitz domains such that $\bigcup_i\overline\omega_i=\overline\omega$,
then for $|s|\leq1$ we have
\begin{equation}\label{e:sobolev-add}
\sum_i \|u\|_{H^s(\omega_i)}^2 \leq C_s \|u\|_{H^s(\omega)}^2,
\qquad\textrm{and}\qquad
\|v\|_{\tilde{H}^s(\omega)}^2 \leq C_s \sum_i \|v\|_{\tilde{H}^s(\omega_i)}^2,
\end{equation}
for $u\in H^{s}(\omega)$ and $v\in \tilde{H}^{s}(\omega)$,
with the constant $C_s>0$ depending only on $s$,
and the norms for $s<0$ defined by duality, cf. \citet*{AMT99} and \cite{vPet89}.
Another interesting norm on $H^{s}(\omega)$ for $s\in(0,1)$, is the \emph{Slobodeckij norm}
\begin{equation}\label{e:slobodeckij}
\|v\|_{s,\omega}^2 = \|v\|_{\omega}^2 +|v|_{s,\omega}^2,
\qquad\textrm{with}\qquad
|v|_{s,\omega}^2 = \iint\limits_{\omega\times\omega} \frac{|v(x)-v(y)|^2}{|x-y|^{2+2s}}\exd x \exd y,
\end{equation}
where $\|v\|_{\omega}$ denotes the $L^2$-norm on $\omega$.
It is immediate that if $\omega_1,\ldots,\omega_k$ are disjoint and Lebesgue measurable sets such that $\bigcup_i\omega_i\subseteq\omega$,
then for $s>0$ we have
\begin{equation}\label{e:slobodeckij-add}
\sum_i |v|_{s,\omega_i}^2 \leq |v|_{s,\omega}^2,
\qquad v\in H^{s}(\omega),
\end{equation}
which is the counterpart of the first inequality in \eqref{e:sobolev-add}.
The seminorms $|\cdot|_{H^s(\omega)}$ and $|\cdot|_{s,\omega}$ are equivalent, with the equivalence constants depending only on $n$, $s$, the Lipschitz constant of $\omega$, and as usual the particulars of the definition of $H^s(\Omega)$. This fact can be proven as in \cite{McL00} by relating the both seminorms to the canonical norms defined through the norm on $H^s(\Omega)$.
A more direct way to relate the two seminorms would be to first connect the Slobodeckij seminorm with the Besov-style seminorm defined by moduli of continuity, 
and then use the equivalence between the moduli of continuity and the $K$-functional, established in \cite{JS77}.
Note that one has to inspect (and adapt) the proofs in \cite{McL00} and \cite{JS77} to reveal the relevant information on the equivalence constants.
It should be emphasized that since we shall be dealing with an infinite collection of domains, in partiuclar of sizes that can shrink to zero,
when using norm equivalences one must be careful about ensuring a control over the equivalence constants.
In the current setting, the equivalence constants do not depend on the size of $\omega$,
and the Lipschitz constants are controlled by restricting the class of domains to shape regular triangles, see below \eqref{e:shape-regular}.
This can also be seen directly from the fact that our refinement procedures lead to only finitely many equivalence classes of triangles.
In the proofs, we make an effort to use the interpolatory norms for as long as possible,
and so to apply the norm equivalence only when necessary.

\subsection{The operator equation}
In the following, we fix $\Gamma\subseteq\Omega$ to be the whole of
$\Omega$, or a connected open set whose boundary consists of curved polygons.
This will be the domain on which we consider our main operator equation
\begin{equation}
  \label{e:op1}
  Au=f,
\end{equation}
where $f\in H^{-t}(\Gamma)$, and $A:\tilde{H}^{t}(\Gamma)\to
H^{-t}(\Gamma)$ is a linear
homeomorphism.
We assume that the operator $A$ is self-adjoint and satisfies
\begin{equation}
  \label{e:elliptic}
  \langle Av,v\rangle\geq\alpha\|v\|_{\tilde{H}^t(\Gamma)}^2,
  \qquad
  \|Av\|_{H^{-t}(\Gamma)}\leq\beta\|v\|_{\tilde{H}^t(\Gamma)},
\end{equation}
for $v\in\tilde{H}^t(\Gamma)$,
with some constants $\alpha>0$ and $\beta>0$, where
$\langle\cdot,\cdot\rangle$ is the duality pairing between
$H^{-t}(\Gamma)$ and $\tilde{H}^t(\Gamma)$.
We introduce the \emph{energy norm}
$\len\cdot\ren=\langle A \cdot,\cdot\rangle^{1/2}$,
and note that it is equivalent to the $\tilde{H}^t(\Gamma)$-norm
\begin{equation}
  \label{e:eneql2}
  \sqrt\alpha\|\cdot\|_{\tilde{H}^t(\Gamma)}\leq\len\cdot\ren\leq\sqrt\beta\|\cdot\|_{\tilde{H}^t(\Gamma)}.
\end{equation}
We also have the norm equivalence
\begin{equation}\label{e:res-eq}
\alpha\|\cdot\|_{\tilde{H}^t(\Gamma)}\leq\|A\cdot\|_{H^{-t}(\Gamma)}\leq\beta\|\cdot\|_{\tilde{H}^t(\Gamma)},
\end{equation}
which is the basis of all residual based error estimation techniques.


Suppose that a closed linear subspace $S\subset \tilde{H}^t(\Gamma)$ is given.
Then the \emph{Galerkin approximation} $u_S\in S$ of $u$ from the space $S$ is characterized by
\begin{equation}\label{e:Galerkin}
\langle Au_S,v\rangle = \langle f,v\rangle,
\qquad
\forall v\in S.
\end{equation}
We have the \emph{Galerkin orthogonality} 
\begin{equation}\label{e:gal-orth}
\len u-u_S \ren^2 + \len u_S-v \ren^2 = \len u-v \ren^2,
\qquad
\forall v\in S,
\end{equation}
which implies that $u_S$ is the best approximation of $u$ from $S$
in the energy norm,
and that the Galerkin approximation is stable:
\begin{equation}\label{e:gal-proj}
  \sqrt\alpha\|u_S\|_{\tilde{H}^t(\Gamma)}\leq\len u_S\ren\leq\len u\ren\leq\frac{1}{\sqrt\alpha}\|f\|_{H^{-t}(\Gamma)}.
\end{equation}
In light of \eqref{e:res-eq}, the residual $r_S=f-Au_S$ is equivalent to the error:
\begin{equation}\label{e:res}
\alpha\|u-u_S\|_{\tilde{H}^t(\Gamma)}\leq\|r_S\|_{H^{-t}(\Gamma)}\leq\beta\|u-u_S\|_{\tilde{H}^t(\Gamma)}.
\end{equation}
In the sense that the residual is a computable quantity that gives bounds on
the true error in terms of this equivalence,
the first inequality in \eqref{e:res} is an example of a \emph{global upper
bound},
while the second one is that of a \emph{global lower bound}.
Upper and lower bounds in this context are also called \emph{reliability} and
\emph{efficiency}, respectively.
The central issue in the theory of residual based error indicators is to
somehow {\em localize} the quantity $\|r_S\|_{H^{-t}(\Gamma)}$ so as to obtain a useful information on which part of $\Gamma$ needs more attention.

\subsection{Triangulations}
We study the Galerkin approximation by
piecewise constant or continuous piecewise linear functions on adaptively generated triangulations of
the manifold $\Gamma$.
Let us now fix some notations and terminologies related to this
discretization.
An open subset of $\Gamma$ is called a \emph{(surface) triangle}
if its closure is diffeomorphic to a flat triangle, 
and the latter is said to be the \emph{reference} of the former.
Assuming that a choice is made of a reference for each surface
triangle,
notions related to flat triangles can be planted onto surface
triangles through their references.
For instance, in the following, straight lines, midpoint, etc., should
be understood in terms of the reference triangles.
We call a collection $P$ of surface triangles a {\em partition} of $\Gamma$
if  $\overline\Gamma=\bigcup_{\tau\in P}\overline\tau$, and
$\tau\cap\sigma=\varnothing$ for any two different $\tau,\sigma\in P$. 
For refining the meshes we mainly use the so called {\em newest vertex
  bisection} algorithm, 
which we describe now for the reader's convenience.
General discussions on this algorithm can be found e.g., in \citet*{BDD04,Stev08}.
We assume that with any triangle $\tau$ comes its
associated newest vertex $v(\tau)$, so that when it is needed to
be refined, $\tau$ is subdivided
into two triangles by connecting $v(\tau)$ with the midpoint of
the edge opposite to it.
The midpoint used in the bisection is now the newest vertex of the both
new triangles.
The 2 new triangles so obtained are called the {\em children} of
$\tau$, and the refinement of $\tau$ is just the collection of
its children. 
The children of a triangle inherit the reference map from their parent.
A partition $P'$ is called a refinement of $P$ and
denoted $P\preceq P'$ if $P'$
can be obtained by replacing zero or more $\tau\in P$ by its
children, or by a recursive application of this procedure.
This procedure is extended to higher dimensions in \cite{Stev08}.


A partition $P$ is said to be {\em conforming} if any vertex $v$ of a
triangle in $P$ is a vertex of all $\tau\in P$ whose closure contains $v$. 
Throughout this paper we consider only partitions that are
refinements of some {\em fixed} conforming partition $P_0$ of
$\Gamma$. 
We require that the references of the initial partition $P_0$ be so
that for any pair of surface triangles that share a common edge,
the parameterizations from the reference triangles to the common edge
are equal up to the composition with an affine map.
The motivation for this is that we want the refinements on both
triangles to agree on the common edge.
A choice of refinement procedure immediately leads to 
the set $[P_0]$ of all partitions that are refinements of $P_0$.
We assume that the family $[P_0]$ is {\em shape regular}, meaning that 
\begin{equation}\label{e:shape-regular}
\shape=\sup \left\{ \frac {h_\tau^n}{\vol(\tau)} : \tau\in P,\,P\in[P_0]
  \right\} < \infty,
\end{equation}
where $h_\tau=\diam(\tau)$.
Both the newest vertex bisection and the red refinement procedures
produce shape regular partitions.
The set $[P_0]$ is too large in the sense that often we are interested
in a certain subset of it that has a good analytic property,
e.g., we want to single out the conforming partitions from $[P_0]$.
Exactly what subset we want depends on the particular setting,
and at this level of generality we simply assume that there is a
subset $\adm(P)\subset[P_0]$ called the family of \emph{admissible partitions},
which is \emph{graded} (or \emph{locally quasi-uniform}, or have the \emph{K-mesh
property}),
i.e., 
\begin{equation}\label{e:K-mesh}
\grade=\sup \left\{\frac {h_\sigma}{h_\tau}: \sigma,\tau\in
P,\,\overline\sigma\cap\overline\tau\neq\varnothing,\,P\in\adm(P_0)\right\} < \infty.
\end{equation}
For example, if $n\geq2$ then the conforming refinements of $P_0$ produced by the
newest vertex bisection are locally quasi-uniform.
If $n=1$, 
we \emph{define} the
admissible partitions to be the ones for which the quantity under the
supremum in \eqref{e:K-mesh} is bounded by a fixed number.
Finally, note that the shape regularity and local quasi-uniformity
together imply \emph{local finiteness}, meaning that the number of triangles
meeting at any given point is bounded by a constant that depends only
on $\shape$, $\grade$, and $n$.

The admissible partitions are the only ones that are ``visible'' to the analytic components of the algorithms.
Hence from both analytic and algorithmic perspectives,
it is convenient to separate the ``analytic'' components that only see admissible partitions,
from the ``combinatoric'' components that make possible the illusion that there are only admissible partitions.
These ``combinatoric'' issues are common to both the FEM and BEM, and mostly settled.
We take them into account by assuming the existence of a couple of operations on admissible partitions.
The first operation is that of refinement, which in practice is implemented by a usual naive refinement possibly producing a non-admissible partition,
followed by a so-called {\em completion} procedure.
Given a partition $P\in\adm(P_0)$ and a set $R\subset P$ of its triangles,
the refinement procedure produces $P'\in\adm(P_0)$,
such that $P\setminus P'\supseteq R$, i.e., the triangles in $R$ are refined at least once.
Let us denote it by $P'=\refine(P,R)$.
We assume the following on its efficiency:
If $\{P_k\}\subset\adm(P_0)$ and $\{R_k\}$ are sequences such that $P_{k+1}=\refine(P_k,R_k)$
and $R_k\subset P_k$ for $k=0,1,\ldots$,
then
\begin{equation}\label{e:complete}
\#P_k-\#P_0\leq C_c\sum_{m=0}^{k-1} \#R_m,
\qquad k=1,2,\ldots,
\end{equation}
where $C_c>0$ is a constant.
This assumption is justified for newest vertex bisection algorithm in
\cite{BDD04,Stev08},
and demonstrated for a 1D refinement procedure in \cite*{AFFKP12}.

Another notion we need is that of \emph{overlay} of partitions:
We assume that there is an operation
$\oplus:\adm(P_0)\times\adm(P_0)\to\adm(P_0)$ satisfying
\begin{equation}\label{e:overlay}
P\oplus Q\succeq P,
\qquad
P\oplus Q\succeq Q,
\qquad
\textrm{and}
\qquad
\#(P\oplus Q)\leq\#P+\#Q-\#P_0,
\end{equation}
for $P,Q\in\adm(P_0)$.
In the conforming world, $P\oplus Q$ is taken to be the smallest
and common conforming refinement of $P$ and $Q$, for which
\eqref{e:overlay} is demonstrated in \cite{CKNS08}.
For a 1D refinement procedure, a justification is given in \cite{AFFKP12}.

\subsection{Discretization}
Given a partition $P\in[P_0]$, we define the piecewise polynomial space ${S}_{P}^d$ by
\begin{equation}\label{e:fem-space}
{S}_{P}^d=\left\{u\in L^2(\Gamma): u\in C(\Gamma) \textrm{ if }d>0,
\textrm{ and } u|_{\tau}\in \mathbb{P}_{d}\,\forall\tau\in P\right\},
\end{equation}
where $\mathbb{P}_{d}$ denotes the set of polynomials of degree less
than or equal to $d$.
Note that for curved triangles, polynomials are defined through the
reference triangles.
We have $S_P^d\subset H^s(\Gamma)$ for $s<\frac12$ if $d=0$, and $s<\frac32$ if $d>0$, with $|s|\leq\mns$ in both cases.
Of interest to us are only the spaces $S_P^0$ of piecewise constants
and $S_P^1$ of continuous piecewise linears.
We will also employ a slight variation of $S_P^1$,
that is the space $\tilde{S}_P^1$ of
piecewise affine functions that vanish on the boundary of $\Gamma$.

Now we collect some estimates relating different Sobolev norms for finite element spaces and their complements.
We will indicate if the constants involved in the estimates depend on parameters
(such as $\shape$) other than $d$ and $n$.
First of all, we recall Faermann's estimate
\begin{equation}\label{e:slobodeckij-faermann}
\|v\|_{s,\Gamma}^2 \leq C_F \sum_{z\in N_P} |v|_{s,\omega(z)}^2,
\qquad v\in H^{s}(\Gamma),\,v\perp_{L^2} S_P^0,
\quad(s\in[0,1]),
\end{equation}
for all admissible partitions $P\in\adm(P_0)$,
with $C_F=C_F(\shape,\grade)$, cf. \cite{Faer00,Faer02}.
Here $N_P$ is the set of vertices in $P$,
and $\omega(z)=\intr\bigcup_{\{\tau\in
P:z\in\overline\tau\}}\overline\tau$ is called the \emph{star} associated to the vertex
$z$, with ``$\intr$'' denoting the interior.
The same estimate for interpolatory norms has been established in \citet*{CMS01},
by a very flexible technique.
We will be using their technique on several occasions in \S\ref{s:negative}.

For a partition $P\in[P_0]$,
let $h_P\in S_P^0$ be such that $h_P(x)=h_\tau$ for $x\in\tau\in P$.
We introduce the space $H^r(\Gamma,P_0)$ for $r>0$ as the space of
functions $v\in L^2(\Omega)$ with $v|_\tau\in H^r(\tau)$ for every
triangle $\tau\in P_0$.
Now for $v\in H^s(\Gamma,P_0)$ with $s\in[0,1]$, let $v_P\in S_P^0$ be the $L^2$-orthogonal projection of $v$ onto $S_P^0$.
Then we have the direct estimate
\begin{equation}\label{e:direct-l2}
\|v-v_P\|_{\tau} \leq C_Jh_\tau^s|v|_{s,\tau},
\qquad \tau\in P,
\end{equation}
where the constant $C_J=C_J(\shape)$.
An immediate consequence is that
\begin{equation}\label{e:direct-pos}
\|h_P^r(v-v_P)\|_{\gamma}^2\leq C_J^2\sum_{\tau\in Q}h_\tau^{2(r+s)}|v|_{s,\tau}^2,
\end{equation}
for
$Q\subseteq P$, $\gamma=\intr\bigcup_{\tau\in Q}\overline\tau$, and
$r\in\R$.
By using a duality argument, and the bounds \eqref{e:direct-pos} and \eqref{e:slobodeckij-add}, one can show also that
\begin{equation}\label{e:direct-neg}
\|v-v_P\|_{\tilde{H}^{-s}(\gamma)} \leq C_J\|h_P^sv\|_{\gamma},
\qquad
v\in L^2(\Gamma).
\end{equation}

For the continuous piecewise linears,
the $L^2$-projection is nonlocal,
and so for convenience
we will employ a quasi-interpolation operator $Q_P:L^2(\Gamma)\to\tilde{S}_P^1$ that satisfies
\begin{equation}\label{e:direct-quasi}
\|v-Q_Pv\|_{r,\gamma} \leq C_J \left(\max_{\tau\in Q}h_\tau\right)^{s-r}|v|_{s,\omega(\gamma)},
\qquad v\in\tilde{H}^s(\Gamma),
\end{equation}
for $0\leq r\leq s\leq1$ and $P\in\adm(P_0)$,
where $\omega(\gamma)=\intr\bigcup_{\{\sigma\in
P:\overline\sigma\cap
\overline\gamma\neq\varnothing\}}\overline\sigma$,
and if $n\geq2$, $\adm(P_0)$ is understood to be the
conforming partitions created by newest vertex bisections from $P_0$.
Recall also that $\tilde{S}_P^1$ is the subspace of
$S_P^1$ with the homogeneous boundary condition.
By a \emph{quasi-interpolation} we mean
that $(Q_Pv)|_{\tau}=(Q_Pv|_{\omega(\tau)})|_{\tau}$.
Examples of such operators are constructed, e.g., in
\cite{Clem75,SZ90,Osw94,BG98}.
Accordingly, in this setting, the estimate \eqref{e:direct-pos} is replaced by
\begin{equation}\label{e:direct-pos-quasi}
\|h_P^r(v-Q_Pv)\|_{\gamma}^2
\leq C_J^2\sum_{\tau\in Q}h_\tau^{2(r+s)}|v|_{s,\omega(\tau)}^2.
\end{equation}
We stated the estimates \eqref{e:direct-l2} and \eqref{e:direct-quasi} in terms of the Slobodeckij norms,
but the usual way to derive these estimates is by interpolation and norm equivalences,
so in particular the same estimates hold with interpolatory norms.

Let us also recall the inverse estimates
\begin{equation}\label{e:std-inverse}
\|v\|_{H^s(\gamma)} \leq C_B\|h_P^{-s}v\|_{\gamma},
\qquad
v\in S_{P}^{d},
\end{equation}
for $s\in[0,\frac12)\cap[0,\mns]$ if $d=0$, and
for $s\in[0,\frac32)\cap[0,\mns]$ if $d>0$, and
\begin{equation}\label{e:new-inverse}
\|h_P^sv\|_{\gamma} \leq C_B\|v\|_{\tilde{H}^{-s}(\gamma)},
\qquad
v\in S_{P}^{d},
\end{equation}
for $s\in[0,\mns]$,
which hold for admissible partitions $P\in\adm(P_0)$, 
with the constant $C_B=C_B(\shape,\grade)$.
Recall that $\gamma=\intr\bigcup_{\tau\in Q}\overline\tau$ for some $Q\subseteq P$.
The both inequalities are proved in \citet*{DFGHS04} in a more general setting, with a piecewise affine mesh-size function $h_P$.
Nevertheless, by local quasi-uniformity, their results immediately imply \eqref{e:std-inverse}
and \eqref{e:new-inverse} with piecewise constant $h_P$,
since in our setting $h_P$ enters only in a derivative-free fashion.
Finally, we will make crucial use of inequalities of the type
\begin{equation}\label{e:inverse-ineq-gen}
\sum_{\tau\in P}h_\tau^{2(s+t)}\|Av\|_{s,\tau}^2 \leq C_A\|v\|_{\tilde{H}^{t}(\Gamma)}^2,
\qquad
v\in S_{P}^{d},
\end{equation}
that is assumed to hold for admissible partitions $P\in\adm(P_0)$, with $C_A=C_A(A,\shape,\grade)$.
This inequality is somewhat more demanding than the standard inverse estimates since it involves the non-local operator $A$,
and in some sense it requires $A$ to be almost local.
In fact, boundary integral operators have certain locality properties,
which is exploited in our analysis only through this inequality.
We prove it in Theorem \ref{t:inverse-ineq} for a wide class of boundary integral operators,
but in general allowing only $C^{1,1}$ surfaces.
\citet{FKMP11,FKMP11a} prove
\eqref{e:inverse-ineq-gen} for $s=1$, and $A$ equal to the simple
layer potential operator on polyhedral surfaces.

\section{Operators of order zero}
\label{s:zero}

In this section, we focus on the case where the operator $A$ is of order zero,
i.e., the case $t=0$.
This is a nice model case to test our arguments on.
On a practical side, 
this case is a representative of Hilbert-Schmidt operators on general
domains or manifolds, and a {\em model} of boundary integral operators associated
to the double layer potential.
In the latter case, for instance when the surface $\Gamma$ is $C^1$ so that the double layer potential operator is compact, one only has a G{\aa}rding-type inequality instead of
the strict coercivity \eqref{e:elliptic}, 
but we expect that such cases can be handled at the
expense of requiring a sufficiently fine initial mesh, in the spirit
of \cite{MN05} and \cite{Gant08}.

The surface $\Gamma$ can be either closed or open.
We employ the piecewise constants $S_P^0$.
The Galerkin approximation of $u$ from $S_P^0$ is denoted by $u_P\in S_P^0$,
and the corresponding residual --- by $r_P$.
Recall the
notations $\|\cdot\|$ and $\|\cdot\|_\omega$ for the $L^2$-norms on
$\Gamma$ and $\omega\subset\Gamma$, respectively.

The equivalence \eqref{e:res} provides the convenient starting point
\begin{equation}\label{e:global-zero}
\alpha\|u-u_P\|\leq\|r_P\|\leq\beta\|u-u_P\|,
\end{equation}
which suggests us to use local $L^2$-norms of the residual as error indicators.
Below we prove a localized version of \eqref{e:global-zero} with the error replaced by the
difference between two Galerkin approximations.
The simple observations in its proof are the essence
of this paper,
in the sense that the rest of the paper can be thought of as an attempt to exploit their natural consequences
and to extend the arguments to non-zero order operators.
Note that no explicit condition whatsoever is imposed on the locality of $A$.

\begin{lemma}\label{l:res-zero}
Let $P,P'\in[P_0]$ be partitions with $P\preceq P'$,
and let
$\Gamma^*=\intr\bigcup_{\tau\in P\setminus P'}\overline\tau$. Then we have
\begin{equation}\label{e:local-discrete}
  \alpha\|u_{P}-u_{P'}\|
  \leq
  \|r_{P}\|_{\Gamma^*} 
  \leq
  \beta\|u_{P}-u_{P'}\| + 2\|r_{P}-v\|_{\Gamma^*},
\end{equation}
for any function $v\in S_{P'}^0$.
\end{lemma}

\begin{proof}
Recall that ``$\intr$'' denotes the interior,
and that $P\setminus P'=\{\tau\in P:\tau\not\in P'\}$, so
$\Gamma^*$ is the region covered by the refined triangles.
Let $e=u_{P'}-u_{P}$,
and let $e_P\in S_P^0$ be the $L^2$-orthogonal projection of $e$ onto $S_P^0$.
Then from the Galerkin condition \eqref{e:Galerkin}
we get
\begin{equation}
  \langle Ae,e\rangle
  =\langle r_P,e\rangle 
  = \langle r_P,e-e_P\rangle 
  \leq \|r_P\|_{\Gamma^*} \|e-e_P\|_{\Gamma^*}
  \leq \|r_P\|_{\Gamma^*} \|e\|_{\Gamma^*},
\end{equation}
where we have used that $e=e_P$ outside $\Gamma^*$.
This proves the first inequality in \eqref{e:local-discrete}.

To prove the second inequality, let $v\in S_{P'}^0$ be supported in $\Gamma^*$.
Then we have
\begin{equation}\label{e:local-discrete-arg}
\begin{split}
\|v\|_{\Gamma^*}^2
&= \langle v,v\rangle 
= \langle v-r_P,v\rangle + \langle A(u_{P'}-u_P),v\rangle\\
&\leq \left(\|v-r_P\|_{\Gamma^*}+\|A(u_{P'}-u_P)\|_{\Gamma^*}\right)\|v\|_{\Gamma^*},
\end{split}
\end{equation}
implying that
\begin{equation}
\|r_P\|_{\Gamma^*}
\leq \|r_P-v\|_{\Gamma^*}+\|v\|_{\Gamma^*}
\leq 2\|r_P-v\|_{\Gamma^*}+\|A(u_{P'}-u_P)\|,
\end{equation}
which gives the second inequality in \eqref{e:local-discrete}.
\end{proof}

\begin{remark}
The arguments used in the preceding proof are of course inspired by the corresponding finite element theory.
However, especially the argument \eqref{e:local-discrete-arg} for the lower bound seems to have a new flavour,
in that it does not break the action of $A$ up into element-wise (or star-wise) operations,
making it particularly suitable for nonlocal operators.
\end{remark}

We recognize the term $\|r_P-v\|_{\Gamma^*}$ in
\eqref{e:local-discrete} as an oscillation term, 
which measures how much of the residual is captured when we move from
$P$ to $P'$.
It is of interest to control this term.
Since $S_P^0\subset H^r(\Gamma)$ for $r\in(0,\frac12)$, \emph{assuming}
that $A:H^r(\Gamma)\to H^r(\Gamma)$ is bounded for all $r\in(0,\frac12)$, we have $Au_P\in
H^r(\Gamma)$ for $r$ in the same range.
Now, \emph{assuming} that $f\in H^r(\Gamma,P_0)$ for some $r\in(0,\frac12)$,
this ensures $r_P\in H^r(\Gamma,P_0)$, 
therefore from \eqref{e:direct-pos} we have
\begin{equation}
\inf_{v\in S_{P'}^0}\|r_{P}-v\|_{\Gamma^*}^2\leq C_J^2\sum_{\tau\in P\setminus P'}h_\tau^{2r}|r_P|_{r,\tau}^2.
\end{equation}
This suggests us to define the \emph{residual oscillation}
\begin{equation}\label{e:res-osc-zero}
\osc_r(v,P,\omega) := \left(\sum_{\tau\in P,\,\tau\subset\omega}h_\tau^{2r}|f-Av|_{r,\tau}^2\right)^{\frac12},
\end{equation}
for $\omega\subseteq\Gamma$ and $v\in S_P^0$, so that \eqref{e:local-discrete} implies
\begin{equation}\label{e:local-discrete-osc}
  \alpha\|u_{P}-u_{P'}\|
  \leq
  \|r_{P}\|_{\Gamma^*} 
  \leq
  \beta\|u_{P}-u_{P'}\| + 2C_J\,\osc_r(u_P,P,\Gamma^*).
\end{equation}
If the oscillation term is sufficiently small,
the difference between two discrete solutions is completely controlled by a local $L^2$-norm of the residual.
In this sense, the first inequality in \eqref{e:local-discrete-osc} is an example of
a \emph{local discrete upper
bound},
while the second one is that of a \emph{local discrete lower bound}.
These bounds are also called \emph{local discrete reliability} and \emph{local
discrete efficiency}, respectively.

In the following, we will \emph{fix} some $r\in(0,1]$,
and assume\footnote{In view of \eqref{e:res-osc-zero}, one has a computational advantage if $r=1$, since there would be no fractional norms.} 
that $f\in H^r(\Gamma,P_0)$.
Note that this is in the same spirit as assuming $f\in L^2$ in the
context of second order elliptic equations, even though there the weak
formulation is well-posed for
$f\in H^{-1}$.
We will use the convenient abbreviation
$\osc_r(P,\omega)=\osc_r(u_P,P,\omega)$.
The next lemma collects crucial properties of the oscillation
$\osc_r(P,\Gamma)$ and the combination $\len u-u_P \ren^2 +
\osc_r(P,\Gamma)^2$,
\emph{assuming} an inverse-type inequality, which shall be verified in \S\ref{s:inverse}.
As it turns out, this inverse-type inequality is also sufficient to guarantee the finiteness of oscillation \eqref{e:res-osc-zero}.
The quantity $\len u-u_P \ren^2 +
\osc_r(P,\Gamma)^2$, the counterpart of the \emph{total error} in \citet*{CKNS08}, will be the main character in our subsequent analysis.

\begin{lemma}\label{l:osc-red-zero}
Assume that
\begin{equation}\label{e:inverse-ineq-zero}
\sum_{\tau\in P}h_\tau^{2r}|Av|_{r,\tau}^2 \leq C_A \len v \ren^2,
\qquad
v\in S_{P}^0,
\end{equation}
for $P\in\adm(P_0)$, with the constant $C_A=C_A(A,\shape,\grade)$.
Then the oscillation \eqref{e:res-osc-zero} is finite for any $v\in S_P^0$ and $P\in\adm(P_0)$.
Moreover, the followings hold.
\begin{enumerate}[a)]
\item
There is a constant $C_G>0$ such that
\begin{equation}\label{e:gal-opt-zero}
\len u-u_P \ren^2 + \osc_r(u_P,P,\Gamma)^2 \leq C_G \inf_{v\in S_P^0} \left( \len u-v \ren^2 + \osc_r(v,P,\Gamma)^2 \right),
\end{equation}
for any $P\in\adm(P_0)$.
\item
There exists a constant $\lambda>0$ such that
\begin{equation}\label{e:osc-red-zero}
\osc_r(P',\Gamma)^2 
\leq 
(1+\delta) \osc_r(P,\Gamma)^2 
- \lambda(1+\delta)\osc_r(P,\Gamma^*)^2 
+ C_\delta \len u_P-u_{P'} \ren^2,
\end{equation}
for any $P,P'\in\adm(P_0)$ with $P\preceq P'$,
and for any $\delta>0$, with $C_\delta$ depending on $\delta$,
where $\Gamma^*=\intr\bigcup_{\tau\in P\setminus P'}\overline\tau$.
\item
Let $P,P'\in\adm(P_0)$ be such that $P\preceq P'$,
and let
$\Gamma^*=\intr\bigcup_{\tau\in P'\setminus P}\overline\tau$.
If, for some $\mu\in(0,\frac12)$ it holds that
\begin{equation}\label{e:err-red-zero-l}
\len u-u_{P'}\ren^2+\osc_r(P',\Gamma)^2 \leq\mu\left(\len u-u_{P}\ren^2+\osc_r(P,\Gamma)^2\right),
\end{equation}
then with $\theta=\frac{\alpha^3(1-2\mu)}{\beta^3(1+2C_A)}$, we have
\begin{equation}\label{e:dorfler-zero-l}
\|r_P\|_{\Gamma^*}^2 +\osc_r(P,\Gamma^*)^2 \geq \theta \left( \|r_P\|^2 +
  \osc_r(P,\Gamma)^2 \right).
\end{equation}
\end{enumerate}
\end{lemma}

\begin{remark}
The estimate \eqref{e:inverse-ineq-zero} is proved in Section
\ref{s:inverse} for a general class of singular integral operators of
order zero, under the hypothesis
that $\hat{A}:H^{\sigma}(\Omega)\to H^{\sigma}(\Omega)$ is bounded for some
$\sigma>\frac{n}2$ and that $r\in(0,\sigma)$, where $\hat{A}$
is an extension of $A$ to $\Omega$ if $\Gamma$ is open,
and $\hat{A}=A$ if $\Gamma=\Omega$.
Recalling that $\Omega$ is a $C^{\mns-1,1}$-manifold,
the condition on $\sigma$ translates to $\mns>\frac{n}2$, see Remark
\ref{r:bdd-bdry-int} for details.
Therefore Lipschitz curves are allowed, but for $n=2$ we need a $C^{1,1}$
surface.
Note that there is no condition on the regularity of the boundary of
$\Gamma$ other than the Lipschitz condition, when $\Gamma$ is an open
surface.
In any case, we believe that the restriction $\mns>\frac{n}2$ is an
artifact of our proof,
and anticipating future weakenings of this restriction, 
the rest of this section will be presented so that it depends only on the
assumption \eqref{e:inverse-ineq-zero}.
\end{remark}

\begin{proof}[Proof of Lemma \ref{l:osc-red-zero}]
Let $w\in S_P^0$ satisfy
\begin{equation}
\len u-w \ren^2 + \osc_r(w,P,\Gamma)^2 \leq 2 \inf_{v\in S_P^0} \left( \len u-v \ren^2 + \osc_r(v,P,\Gamma)^2 \right).
\end{equation}
Then from the definition of oscillation, we have
\begin{equation}
\begin{split}
\osc_r(u_P,P,\Gamma)^2
&\leq 2\sum_{\tau\in P}h_\tau^{2r}|f-Aw|_{r,\tau}^2 + 2\sum_{\tau\in P}h_\tau^{2r}|A(w-u_P)|_{r,\tau}^2\\
&\leq 2\,\osc_r(w,P,\Gamma)^2 + 2C_A \len w-u_P \ren^2,
\end{split}
\end{equation}
where in the last step we have used inverse inequality
\eqref{e:inverse-ineq-zero}.
This gives
\begin{equation}
\begin{split}
\len u-u_P \ren^2 + \osc_r(u_P,P,\Gamma)^2
&\leq \len u-u_P \ren^2 +  2C_A \len w-u_P \ren^2 + 2\,\osc_r(w,P,\Gamma)^2,
\end{split}
\end{equation}
and upon using the Galerkin orthogonality $\len u-u_P \ren^2 + \len
w-u_P \ren^2 = \len u-w \ren^2$,
we obtain \eqref{e:gal-opt-zero} with, say $C_G=4(1+C_A)$.

Now we turn to b). 
With $\lambda_0\in(0,1)$ being the contraction factor for $h_\tau$ when $\tau$ is refined once, we have
\begin{equation}
\sum_{\tau\in P'\setminus P}h_\tau^{2r}|r_{P}|_{r,\tau}^2
\leq \sum_{\tau\in P\setminus P'} \lambda_0^{2r} h_\tau^{2r}|r_{P}|_{r,\tau}^2
= \lambda_0^{2r} \osc_r(P',\Gamma^*)^2.
\end{equation}
Using this in
\begin{equation}
\osc_r(P',\Gamma)^2
\leq (1+\delta)\sum_{\tau\in P'}h_\tau^{2r}|r_{P}|_{r,\tau}^2 + (1+\frac1\delta) \sum_{\tau\in P'}h_\tau^{2r}|A(u_{P'}-u_{P})|_{r,\tau}^2,
\end{equation}
and by using the inverse inequality \eqref{e:inverse-ineq-zero} on the
last term, we establish \eqref{e:osc-red-zero}.

For c), from \eqref{e:err-red-zero-l} we infer
\begin{equation}
\begin{split}
 (1-2\mu)(\len u-u_{P}\ren^2 + \osc_r(P,\Gamma)^2) 
&\leq
\len u_{P}-u_{P'}\ren^2 + \osc_r(P,\Gamma)^2 - 2\,\osc_r(P',\Gamma)^2\\
&\leq
(1+2C_A)\len u_{P}-u_{P'}\ren^2 + \osc_r(P,\Gamma^*)^2\\
&\leq
\frac{\beta(1+2C_A)}{\alpha^2}\|r_P\|_{\Gamma^*} + \osc_r(P,\Gamma^*)^2,
\end{split}
\end{equation}
where we have used the Galerkin orthogonality,
the estimate
\begin{equation}
\osc_r(P,\Gamma\setminus\Gamma^*)^2
\leq 2\, \osc_r(P',\Gamma\setminus\Gamma^*)^2
 + 2 C_A \len u_P-u_{P'}\ren^2,
\end{equation}
the local discrete upper bound in \eqref{e:local-discrete-osc}, 
and the norm equivalence \eqref{e:eneql2}.
The proof is completed upon noting that
\begin{equation}
\|r_P\|^2
\leq
\gamma\len u-u_{P}\ren^2,
\end{equation}
where $\gamma=\beta^2/\alpha\geq1$, which 
is a combination of the global lower bound in \eqref{e:res}, and the norm equivalence \eqref{e:eneql2}.
\end{proof}

Once we have the preceding results,
and given the techniques developed in \cite{Stev07} and \cite{CKNS08},
it is a fairly straightforward matter to obtain geometric error reduction and
quasi-optimality for an adaptive method such as the ones considered here.
Nevertheless, we include detailed proofs for convenience of the
reader.
Our first stop is the contraction property of the adaptive method.

\begin{proposition}\label{p:cont-zero}
Let the assumption \eqref{e:inverse-ineq-zero} of Lemma \ref{l:osc-red-zero} hold.
Let $P,P'\in\adm(P_0)$ be admissible partitions with $P\preceq P'$,
and let
$\Gamma^*=\intr\bigcup_{\tau\in P\setminus P'}\overline\tau$. 
Suppose, for some $\theta\in(0,1]$ that
\begin{equation}\label{e:dorfler-zero}
\|r_P\|_{\Gamma^*}^2 + \osc_r(P,\Gamma^*)^2 \geq\theta\left(\|r_P\|^2+\osc_r(P,\Gamma)^2\right).
\end{equation}\label{e:err-red-zero}
Then there exist constants $\gamma\geq0$ and $\mu\in(0,1)$ such that
\begin{equation}
\len u-u_{P'}\ren^2+\gamma\,\osc_r(P',\Gamma)^2 \leq\mu\left(\len u-u_{P}\ren^2+\gamma\,\osc_r(P,\Gamma)^2\right).
\end{equation}
\end{proposition}

\begin{proof}
The property \eqref{e:dorfler-zero}, together with
the global upper bound in \eqref{e:res}, the local discrete lower
bound in \eqref{e:local-discrete-osc}, and the norm equivalence \eqref{e:eneql2}, gives
\begin{equation}\label{e:zero-prop-2}
\theta\left(\frac{\alpha^2}{\beta}\len u-u_P\ren^2+\osc_r(P,\Gamma)^2\right) \leq \frac{2\beta^2}{\alpha}\len u_P-u_{P'}\ren^2 + 2(2C_J+1)^2\osc_r(P,\Gamma^*)^2.
\end{equation}
Then for any $\gamma\geq0$, we combine Lemma \ref{l:osc-red-zero} b) and \eqref{e:zero-prop-2} to infer
\begin{multline}
\len u-u_{P'}\ren^2+\gamma\,\osc_r(P',\Gamma)^2
= \len u-u_{P}\ren^2-\len u_{P}-u_{P'}\ren^2+\gamma\,\osc_r(P',\Gamma)^2\\
\leq \len u-u_{P}\ren^2 - (1-\gamma C_\delta) \len u_{P}-u_{P'}\ren^2
+\gamma(1+\delta)\osc_r(P,\Gamma)^2 - \gamma\lambda(1+\delta)\osc_r(P,\Gamma^*)^2\\
\leq \left(1-\frac{\theta\alpha^3 (1-\gamma C_\delta)}{2\beta^2}\right) \len u-u_{P}\ren^2
+ \left(\gamma(1+\delta)-\frac{\theta\alpha (1-\gamma C_\delta)}{2\beta^2}\right)\osc_r(P,\Gamma)^2\\
 + \left(\frac{\alpha(2C_J+1)^2 (1-\gamma C_\delta)}{\beta^2} - \gamma\lambda(1+\delta)\right)\osc_r(P,\Gamma^*)^2\\
=: \mu_1 \len u-u_{P}\ren^2 + \mu_2\gamma\,\osc_r(P,\Gamma)^2 + \mu_3\,\osc_r(P,\Gamma^*)^2.
\end{multline}
We choose $\gamma$, depending on $\delta>0$, so that $\mu_3=0$, i.e., so that
\begin{equation}
\frac1\gamma = C_\delta + \frac{\lambda\beta^2(1+\delta)}{\alpha(2C_J+1)^2}.
\end{equation}
With this choice and for $\delta>0$ sufficiently small, we have
\begin{equation}
\mu_1=1 - \frac{\alpha^2\gamma\lambda\theta(1+\delta)}{2(2C_J+1)^2}<1,
\end{equation}
and
\begin{equation}
\mu_2=(1+\delta)\left(1 - \frac{\lambda \theta}{2(2C_J+1)^2}\right)<1,
\end{equation}
establishing the proof.
\end{proof}

To be explicit, the importance of the preceding proposition is the
following.
Suppose that we have an admissible partition $P$
and the discrete solution $u_P\in S_P^0$.
Then the local residual norm $\|r_P\|_\tau$ and the local oscillation
$\osc_r(P,\tau)$ are computable for $\tau\in P$,
and by selecting a set $R\subset P$ of triangles such that
$\Gamma^*=\intr\bigcup_{\tau\in R}\overline\tau$ satisfies
\eqref{e:dorfler-zero} for some $\theta\in(0,1)$,
and then by finding an admissible refinement $P'$ of $P$ such that
$P\setminus P'\supseteq R$,
i.e., that each triangle in $R$ is refined at least once,
we can guarantee $e(P')\leq\mu e(P)$ for some $\mu<1$, where
$e(Q)=\len u-u_{Q}\ren^2+\gamma\,\osc_r(Q,\Gamma)^2$.
By repeatedly applying this procedure, we can ensure convergence $\len
u-u_{P}\ren\to0$ as $P$ runs over the partitions generated by the
algorithm.
This however, does not say anything about the growth of $\#P$, a fundamental
question that will be addressed below.

For $u\in L^2(\Gamma)$ the solution of $Au=f$ with $f\in H^r(\Gamma,P_0)$,
and for $P\in[P_0]$,
we define
\begin{equation}\label{e:dist-r-zero}
\dist_r(u,S_P^0) = \min_{v\in S_P^0} \left( \len u-v \ren^2+\osc_r(v,P,\Gamma)^2 \right)^{\frac12},
\end{equation}
where the implicit dependence of oscillation on $r$ has been made
explicit.
Note that the minimum exists since $S_P^0$ is finite dimensional.
Furthermore, for $\eps>0$ we define 
\begin{equation}
\card_r(u,\eps) = \min\{\#P-\#P_0:P\in\adm(P_0),\,\dist_r(u,S_P^0)\leq\eps\}.
\end{equation}
Hence $\card_r(u,\eps)$ is in certain sense the cardinality of a
smallest admissible partition $P$ that is able to support a function that is
within an $\eps$ distance from $u$.
Note that $\card_r(u,\eps)$ is finite for any $\eps>0$, since from the
discussion in the preceding paragraph, there is a sequence of
(finite) partitions $\{P_k\}\subset\adm(P_0)$ with $\dist_r(u,S_{P_k}^0)\to0$.

The following proposition shows an optimal way to select the triangles to be
refined.
The procedure is known as \emph{D\"orfler's marking strategy} in literature.

\begin{proposition}\label{p:card-zero}
Let the assumption \eqref{e:inverse-ineq-zero} of Lemma \ref{l:osc-red-zero} hold.
Let $P\in\adm(P_0)$,
and let $\theta\in(0,\theta^*)$ with
$\theta^*=\frac{\alpha^3}{\beta^3(1+2C_A)}$.
Suppose that $R\subseteq P$ is a subset whose cardinality is minimal up to a constant factor $\kappa\geq1$,
among all $R\subseteq P$ satisfying
\begin{equation}
\|r_P\|_{\Gamma^*(R)}^2 + \osc_r(P,\Gamma^*(R))^2 \geq\theta\left(\|r_P\|^2+\osc_r(P,\Gamma)^2\right).
\end{equation}
where $\Gamma^*(R)=\intr\bigcup_{\tau\in R}\overline\tau$.
Then we have
\begin{equation}
\#R\leq \kappa\,\card_r(u,\eps),
\end{equation}
where $\eps$ is defined by
\begin{equation}\label{e:def-eps-zero}
\eps^2 = \frac{\theta^*-{\theta}}{2C_G \theta^*} \left( \len u-u_{P}\ren^2+\osc_r(P,\Gamma)^2 \right).
\end{equation}
\end{proposition}

\begin{proof}
Let us introduce the abbreviation
\begin{equation}\label{e:def-e-zero}
e(Q)=\len u-u_{Q}\ren^2+\osc(Q,\Gamma)^2,
\qquad
Q\in[P_0],
\end{equation}
and let $P_\eps\in\adm(P_0)$ be such that
\begin{equation}
\#P_\eps-\#P_0\leq\card_r(u,\eps),
\qquad
\textrm{and}
\qquad
e(P_\eps)\leq\eps^2.
\end{equation}
Then for $\tilde P = P \oplus P_\eps$, from Lemma \ref{l:osc-red-zero}
a) we have
\begin{equation}
e(\tilde P)\leq C_G e(P_\eps) \leq C_G \eps^2 = \mu e(P),
\end{equation}
where $\mu=\frac{\theta^*-{\theta}}{2\theta^*}\in(0,\frac12)$,
and so an application of Lemma \ref{l:osc-red-zero} c) gives
\begin{equation}
\|r_P\|_{\Gamma^*(P\setminus\tilde{P})}^2 + \osc_r(P,\Gamma^*(P\setminus\tilde{P}))^2 \geq \theta\left(\|r_P\|^2+\osc_r(P,\Gamma)^2\right).
\end{equation}
Recalling that $R$ minimizes $\#R$ up to the constant
factor $\kappa$ among all subsets $P\setminus\tilde P$ satisfying the preceding inequality,
we infer that $\#R\leq \kappa\,\#(P\setminus \tilde P)$,
and taking into account that $\#(P\setminus \tilde P)\leq\#\tilde P -
\#P$ and the estimate in \eqref{e:overlay}, we get
\begin{equation}
\#R \leq \kappa (\#\tilde P-\#P) \leq \kappa (\#P_\eps-\#P_0) \leq\kappa\,\card_r(u,\eps),
\end{equation}
completing the proof.
\end{proof}

We have almost all the ingredients to give a bound on the growth of
$\#P$,
and to discuss whether this growth rate is optimal in one or another sense.
Let us start by making the term ``adaptive BEM'' precise.

\vspace{2mm}
\begin{algorithm}[H]
\SetKwInOut{Params}{parameters}
\SetKwInOut{Output}{output}
\SetKwFor{For}{for}{do}{endfor}
\Params{conforming partition $P_0$, and $\theta\in[0,1]$}
\Output{$P_k\in\adm(P_0)$ and $u_k\in S_{P_k}^0$ for all $k\in\N_0$}
\BlankLine
\For{$k=0,1,\ldots$}{
Compute $u_k\in S_{P_k}^0$ as the Galerkin approximation of $u$ from
$S_{P_k}^0$\;\nllabel{a:galsolve}
Identify a minimal (up to a constant factor) set $R_k\subset P_k$ of triangles satisfying
\begin{equation}
\|r_k\|_{\Gamma^*}^2 + \osc(P_k,\Gamma^*)^2 \geq\theta\left(\|r_k\|^2+\osc(P_k,\Gamma)^2\right),
\end{equation}
where $r_k=f-Au_k$ and $\Gamma^*=\intr\bigcup_{\tau\in R_k}\overline\tau$\;\nllabel{a:mark}
Set $P_{k+1}=\refine(P_k,R_k)$\;\nllabel{a:make-conf}
}
\caption{Adaptive BEM ($t=0$)}\label{a:abem}
\end{algorithm}
\vspace{2mm}

Now we turn to the issue of convergence rate.
Since Proposition \ref{p:card-zero} gives a bound on $\#R_k$, 
we simply need to use \eqref{e:complete} to get
\begin{equation}
\#P_k-\#P_0\leq \kappa C_c\sum_{m=0}^{k-1} \card_r(u,Ce(P_m)),
\end{equation}
where $C$ is the constant from \eqref{e:def-eps-zero}, and $e(P_m)$ is
as in \eqref{e:def-e-zero}.
On the other hand, Proposition \ref{p:cont-zero} guarantees a geometric
decrease of $e(P_m)$, i.e, we have
\begin{equation}
e(P_k)\leq C\mu^{k-m}e(P_m),
\qquad\textrm{hence}\qquad
\card_r(u,Ce(P_m))\leq\card_r(u,C\mu^{m-k}e(P_k)),
\end{equation}
for some constants $C>0$ and $0<\mu<1$.
Note that $C$ denotes different constants in its different
appearances.
Therefore, if our particular $u\in A^{-1}(H^r(\Gamma,P_0))$ satisfies
\begin{equation}\label{e:inst-opt-cond-zero}
\card_r(u,\lambda\eps)\leq C\lambda^{-1/s}\card_r(u,\eps),
\qquad
\lambda>1,\,\eps>0,
\end{equation}
for some constant $s>0$,
then we would get what can be called \emph{instance optimality}
\begin{equation}
\#P_k-\#P_0\leq C \card_r(u,e(P_k)).
\end{equation}
However, it is not clear how to usefully characterize the set of such
$u$,
so we settle for something more modest.
Instead of \eqref{e:inst-opt-cond-zero}, if we have
\begin{equation}
\card_r(u,\eps)\leq C_u\eps^{-1/s},
\qquad
\eps>0,
\end{equation}
for some constant $s>0$,
then we get what can be called \emph{class optimality}
\begin{equation}
\#P_k-\#P_0\leq C C_u e(P_k)^{-1/s}.
\end{equation}

This motivates us to define the \emph{approximation class} $\mathcal{A}_{r,s}\subset
A^{-1}(H^r(\Gamma,P_0))$ with $s\geq0$ to be the set of $u$ for which
\begin{equation}
|u|_{\mathcal{A}_{r,s}} = \sup_{\eps>0} \left( \card_r(u,\eps)^s\eps \right)<\infty.
\end{equation}
Thus $\mathcal{A}_{r,s}$ is characterized by
\begin{equation}
\card_r(u,\eps)\leq  \eps^{-1/s} |u|_{\mathcal{A}_{r,s}},
\end{equation}
or equivalently, by
\begin{equation}
\min\{\dist_r(u,S_P^0):P\in\adm(P_0),\,\#P-\#P_0\leq N\}\leq N^{-s} |u|_{\mathcal{A}_{r,s}}.
\end{equation}

We have proved the following.

\begin{theorem}\label{t:comp-zero}
Let the assumption \eqref{e:inverse-ineq-zero} of Lemma \ref{l:osc-red-zero} hold,
and in Algorithm \ref{a:abem}, suppose that $\theta\in(0,\theta^*)$ with $\theta^*=\frac{\alpha^3}{\beta^3(1+2C_A)}$.
Let $f\in H^{r}(\Gamma,P_0)$ and $u\in\mathcal{A}_{r,s}$ for some $s>0$.
Then we have
\begin{equation}
\len u-u_k\ren^2+\osc_r(u_k,P_k,\Gamma)^2
\leq
C |u|_{\mathcal{A}_{r,s}}^2 (\#P_k-\#P_0)^{-2s},
\end{equation}
where $C>0$ is a constant.
\end{theorem}

\section{Positive order operators}
\label{s:positive}

In this section, we consider the case $t\in[0,\frac34)$. 
We assume that $\Gamma$ is open, and remark that the closed case can
be treated with similar methods.
Recall that the equation we are dealing with is $Au=f$
with linear homeomorphism $A:\tilde{H}^{t}(\Gamma)\to H^{-t}(\Gamma)$.
For simplicity we consider right hand sides satisfying $f\in L^2(\Gamma)$.
We employ the continuous piecewise affine functions with homogeneous boundary conditions
$S_P:=\tilde{S}_P^1$.
If $n\geq2$,
the admissible partitions will be used synonymous with conforming
triangulations, with the newest vertex bisection
algorithm for refinements.
So we have $S_P\subset\tilde{H}^{2t}(\Gamma)$,
and we make an additional assumption that $A:\tilde{H}^{2t}(\Gamma)\to
L^2(\Gamma)$ is bounded, in order to keep the useful property $f-Av\in L^2(\Gamma)$
for any finite element function $v\in S_P$.
Note that this assumption is satisfied for our main example -- the
hypersingular integral operator.
Note also that if $t>\frac12$ then we need $\mns\geq2$, i.e., we need
the space $\Omega$ in which $\Gamma$ lies to be at least a $C^{1,1}$ manifold.

Starting from this section, we shall often dispense with giving explicit
names to constants, and use the Vinogradov-style notation $X\lesssim
Y$, which means $X\leq C\cdot Y$ with some constant $C$ that is allowed to
depend only on the operator $A$ and on (geometry of) the set of
admissible partitions $\adm(P_0)$.
The following is an extension of Lemma \ref{l:res-zero} to $t\geq0$
and piecewise linear finite element spaces.
In contrast to Lemma \ref{l:res-zero}, note that $\omega^*$ includes a buffer layer of triangles around the
refined triangles, so that $\omega^*=\omega(\Gamma^*)$ with $\Gamma^*$ as
in Lemma \ref{l:res-zero}.
By using the Scott-Zhang quasi-interpolation operator in the proof,
it is possible to do without the buffer layer if $t>\frac12$,
which is however not terribly exciting since the relevant case to us is $t=\frac12$.

We define the {\em oscillation} for $v\in S_P$ by
\begin{equation}\label{e:osc-pos}
\osc(v,P) = \|h_{P}^{t}(f-Av) - h_{P}^{-t}w\|,
\end{equation}
where, with $\hat{P}$ the uniform refinement of $P$, $w=Q_{\hat P}h_P^{2t}(f-Av)\in S_{\hat P}$ is the Cl\'ement interpolator of $h_P^{2t}(f-Av)$, given by
\begin{equation}
 Q_{\hat P}g = \sum_{z\in N_{\hat P}\setminus\partial\Gamma} g_z(z)\phi_z,
\end{equation}
where $N_{\hat P}$ is the set of all nodes in $\hat P$,
$\phi_z\in S_{\hat P}$ is the standard nodal basis function at $z$,
and $g_z\in\mathbb{P}_1$ is the $L^2$-orthogonal projection of $g$ onto the affine functions on the star $\hat\omega(z)$ around $z$ with respect to $\hat P$.

\begin{lemma}\label{l:res-pos}
Let $P,P'\in\adm(P_0)$ with $P\preceq P'$,
and let
$\omega^*=\bigcup_{\tau\in P\setminus P'}\omega(\tau)$. 
Then we have
\begin{equation}\label{e:local-discrete-pos}
  \len u_{P}-u_{P'} \ren
  \lesssim
  \|h_{P}^{t}r_{P}\|_{\omega^*} 
  \lesssim
  \|h_{P}/h_{P'}\|_{\infty}^t \len u_{P}-u_{P'} \ren + \|h_{P}^{t}r_{P}-h_{P}^{-t}v\|_{\omega^*},
\end{equation}
for any $v\in S_{P'}$ with $\supp\,v\subseteq \overline{\omega^*}$, where $\|\cdot\|_{\infty}$ denotes the $L^\infty$-norm.
Moreover, we have the following global bounds
\begin{equation}\label{e:global-pos}
  \len u-u_{P} \ren
  \lesssim
  \|h_{P}^{t}r_{P}\|
  \lesssim
  \|h_{P}/h_{P'}\|_{\infty}^t \len u-u_{P} \ren + \|h_{P}^{t}r_{P}-h_{P}^{-t}v\|,
\end{equation}
for any $v\in S_{P'}$.

Furthermore, for $\gamma\subseteq\Gamma$ and $\hat\omega(\gamma)=\bigcup\{\hat\omega(z):z\in N_{\hat P}\cap\overline{\gamma}\}$,
we have 
\begin{equation}\label{e:est-dom-osc}
\|h_{P}^{t}r_P - h_{P}^{-t}Q_{\hat P}
h_{P}^{2t}r_P\|_{\gamma}^2
\leq
C_Q\|h_{P}^{t}r_P\|_{\hat\omega(\gamma)}^2,
\end{equation}
for some constant $C_Q>0$,
i.e., the estimator dominates the oscillation on the Galerkin solutions.
In particular, we have the equivalence
\begin{equation}
\alpha_1\|h_P^tr_P\|^2
\leq
\len u-u_P\ren^2 + \osc(u_P,P)^2
\lesssim
\|h_P^tr_P\|^2,
\end{equation}
where $\alpha_1>0$ is a contstant.
\end{lemma}

\begin{remark}
The global upper bound (i.e., the first inequality) of
\eqref{e:global-pos} has been established in \citet*{CMPS04}, 
and a similar bound involving local $L^p$-norms on stars appears in
\citet*{NPZ10}.
\end{remark}

\begin{proof}[Proof of Lemma \ref{l:res-pos}]
For $v\in \tilde{H}^t(\Gamma)$ and $v_P=Q_Pv\in S_P$ being the quasi-interpolant
of $v$ as in \eqref{e:direct-pos-quasi}, we have
\begin{equation}
\begin{split}
\langle r_P,v\rangle
&= \langle r_P,v-v_P\rangle
= \langle h_{P}^t r_P,h_{P}^{-t}(v-v_P)\rangle\\
&\leq\|h_{P}^tr_P\| \|h_{P}^{-t}(v-v_P)\|
\lesssim \|h_{P}^tr_P\| \|v\|_{\tilde{H}^t(\Gamma)},
\end{split}
\end{equation}
where in the last step we have used
\begin{equation}
\sum_{\tau\in P} h_\tau^{-2t}\|v-v_P\|_\tau^2
\lesssim 
\sum_{\tau\in P} \|v\|_{H^t(\omega(\tau))}^2
\lesssim 
\|v\|_{\tilde{H}^t(\Gamma)}^2,
\end{equation}
by local finiteness.
This gives the first inequality in \eqref{e:global-pos}.

With $v=u_{P'}-u_{P}$ and $v_P=Q_Pv\in S_P$, we infer
\begin{equation}
\begin{split}
\langle Av,v\rangle
&= \langle r_P,v-v_P\rangle
= \langle r_P,v-v_P\rangle_{\omega^*}\\
&\leq\|h_{P}^tr_P\|_{\omega^*} \|h_{P}^{-t}(v-v_P)\|_{\omega^*}
\lesssim \|h_{P}^tr_P\|_{\omega^*} \|v\|_{\tilde{H}^t(\Gamma)},
\end{split}
\end{equation}
where we have used the fact that $v_P=v$ outside $\omega^*$.
This gives the first inequality in \eqref{e:local-discrete-pos}.

Let either $e=u_{P'}-u_{P}$ or $e=u-u_{P}$.
Then in both cases, for $v\in S_{P'}$ and $\omega\supseteq\supp\, v$, we have
\begin{equation}
\begin{split}
\langle h_{P}^{-t}v,h_{P}^{-t}v\rangle_{\omega}
&= \langle h_{P}^{-2t}v,v\rangle
= \langle h_{P}^{-2t}v-r_P,v\rangle + \langle Ae,v\rangle\\
&= \langle h_{P}^{-t}v-h_{P}^{t}r_P,h_{P}^{-t}v\rangle + \langle Ae,v\rangle\\
&\leq \|h_{P}^{-t}v-h_{P}^{t}r_P\|_{\omega} \|h_{P}^{-t}v\|_{\omega} + \|Ae\|_{H^{-t}(\Gamma)} \|v\|_{\tilde{H}^t(\Gamma)}\\
&\lesssim \|h_{P}^{-t}v-h_{P}^{t}r_P\|_{\omega} \|h_{P}^{-t}v\|_{\omega} + \|Ae\|_{H^{-t}(\Gamma)} \|h_{P'}^{-t}v\|_{\omega},
\end{split}
\end{equation}
so that
\begin{equation}
\|h_{P}^{-t}v\|_{\omega} \lesssim \|h_{P}^{-t}v-h_{P}^{t}r_P\|_{\omega} + \|h_{P}/h_{P'}\|_{\infty}^t\|e\|_{\tilde{H}^t(\Gamma)}.
\end{equation}
From this, we infer
\begin{equation}
\|h_{P}^{t}r_P\|_{\omega} \leq \|h_{P}^{t}r_P - h_{P}^{-t}v\|_{\omega} +
\|h_{P}^{-t}v\|_{\omega} 
\lesssim \|h_{P}^{t}r_P - h_{P}^{-t}v\|_{\omega} + \|h_{P}/h_{P'}\|_{\infty}^t\|e\|_{\tilde{H}^t(\Gamma)},
\end{equation}
proving the both second inequalities in \eqref{e:local-discrete-pos} and \eqref{e:global-pos}.
\end{proof}

\begin{remark}[Saturation assumption]\label{r:pos-sat}
For a large number of non-residual \emph{a posteriori} error estimators for the
hypersingular integral equation on curves, \citet*{EFGP09} proved
that the estimators are equivalent to the global error, with the upper bound
depending on the saturation assumption: 
$\len u - u_P \ren\lesssim\len \hat u_P - u_P \ren$, where $\hat u_P$ is the Galerkin approximation from some
enriched space $\hat S_P\supset S_P$, which is typically the piecewise
linears on the uniform refinement of $P$.
Combining the discrete lower bound with the global upper bound from
Lemma \ref{l:res-pos}, we have
\begin{equation}
  \len u-u_{P} \ren
  \lesssim
  \len u_{P}-u_{P'} \ren + \|h_{P}^{t}r_{P}-h_{P}^{-t}v\|,
\end{equation}
for any $v\in S_{P'}$, where $P'$ is the uniform refinement of $P$.
At least in theory, this confirms the saturation assumption up to an oscillation term.
In practice though, to control the oscillation as defined here, it seems that the residual needs to be computed anyways.
The question of whether such an overhead is tolerable calls for further investigation.
\end{remark}

Let us get back to the residual based error
indicators $\|h_\tau^tr_P\|_\tau$ from Lemma \ref{l:res-pos}.
In view of the results in that lemma, the circumstances are very similar to what happens in the finite element
case, and in particular, the local quantities $\|h_P^tr_P\|_\tau$ as error
indicators will give rise to an adaptive algorithm that converges quasi-optimally in a certain sense.

\begin{lemma}\label{l:osc-red-pos}
Assume that
\begin{equation}\label{e:inverse-ineq-pos}
\|h_P^tAv\|^2 \leq C_A \len v \ren^2,
\qquad
v\in S_{P},
\end{equation}
for $P\in\adm(P_0)$, with the constant $C_A=C_A(A,\shape,\grade)$.
Then the followings hold.
\begin{enumerate}[a)]
\item
There is a constant $C_G>0$ such that
\begin{equation}\label{e:gal-opt-pos}
\len u-u_P \ren^2 + \osc(u_P,P)^2 \leq C_G \inf_{v\in S_P} \left( \len u-v \ren^2 + \osc(v,P)^2 \right),
\end{equation}
for any $P\in\adm(P_0)$.
\item
There exists a constant $\lambda>0$ such that
\begin{equation}\label{e:osc-red-pos}
\|h_{P'}^tr_{P'}\|^2 
\leq 
(1+\delta)\|h_{P}^tr_{P}\|^2 
- \lambda(1+\delta) \|h_{P}^tr_{P}\|_{\Gamma^*}^2 
+ C_\delta \len u_P-u_{P'} \ren^2,
\end{equation}
for any $P,P'\in\adm(P_0)$ with $P\preceq P'$,
and for any $\delta>0$, with $C_\delta$ depending on $\delta$,
where $\Gamma^*=\intr\bigcup_{\tau\in P\setminus P'}\overline\tau$.
\item
Let $P,P'\in\adm(P_0)$ be such that $P\preceq P'$,
and let
$\Gamma^*=\intr\bigcup_{\tau\in P\setminus P'}\overline\tau$.
If, for some $\mu\in(0,\frac12)$ it holds that
\begin{equation}\label{e:err-red-pos-l}
\len u-u_{P'}\ren^2+\osc(u_{P'},P')^2 \leq\mu\left(\len u-u_{P}\ren^2+\osc(u_{P},P)^2\right),
\end{equation}
then with $\theta=\frac{\alpha_1(1-2\mu)}{C_Q+\beta_1(1+2C_AC_Q)}$ and $\omega^*=\bigcup_{\tau\in P\setminus P'}\omega(\tau)$, we have
\begin{equation}\label{e:dorfler-pos-l}
\|h_{P}^tr_{P}\|_{\omega^*}^2 \geq \theta \|h_{P}^tr_{P}\|^2,
\end{equation}
where $\beta_1>0$ is a constant implicit in the local discrete upper bound in \eqref{e:local-discrete-pos},
such that $\len u_P-u_{P'} \ren^2\leq\beta_1\|h_{P}^tr_{P}\|_{\omega^*}^2$.
\end{enumerate}
\end{lemma}

\begin{remark}
The estimate \eqref{e:inverse-ineq-pos} is proved in Section
\ref{s:inverse} for a general class of singular integral operators of
positive order, under the hypothesis
that $\hat{A}:H^{t+\sigma}(\Omega)\to H^{-t+\sigma}(\Omega)$ is bounded for some
$\sigma>\max\{\frac{n}2,t\}$, where $\hat{A}$
is an extension of $A$ to $\Omega$.
Recalling that $\Omega$ is a $C^{\mns-1,1}$-manifold,
and assuming that $t\leq\frac{n}2$,
the condition on $\sigma$ translates to $\mns>\frac{n}2+t$, see Remark
\ref{r:bdd-bdry-int} for details.
Therefore for $n=1$ and $n=2$ we need a $C^{1,1}$
curve or a surface, respectively, assuming that $\frac12\leq t<1$.
Note that the boundary of $\Gamma$ is allowed to be Lipschitz.
Anticipating future developments on \eqref{e:inverse-ineq-pos}, 
the rest of this section is presented so as to depend only on the
assumption \eqref{e:inverse-ineq-pos}.
\end{remark}

\begin{proof}[Proof of Lemma \ref{l:osc-red-pos}]
For $v\in S_P$ we have
\begin{equation}
\begin{split}
\osc(u_P,P) 
&= \|h_{P}^{t}r_P - h_{P}^{-t}Q_{\hat P} h_{P}^{2t}r_P\|\\
&\leq \osc(v,P) + \|h_{P}^{t}A(v-u_P) - h_{P}^{-t}Q_{\hat P} h_{P}^{2t}A(v-u_P)\|\\
&\leq \osc(v,P) + C_Q^{1/2}\|h_{P}^{t}A(v-u_P)\|
\leq \osc(v,P) + C_Q^{1/2}C_A^{1/2}\len v-u_P \ren.
\end{split}
\end{equation}
With this estimate at hand, part a) of the lemma can be proven in
exactly the same way as Lemma \ref{l:osc-red-zero} a).

Part b) follows from
\begin{equation}
\|h_{P'}^tr_{P'}\|^2
\leq
(1+\delta)\|h_{P'}^tr_{P}\|^2 + C_A (1+\delta^{-1})\len u_{P}-u_{P'} \ren^2,
\end{equation}
and the fact that $h_{P'}\leq\lambda_0h_{P}$ on $\Gamma^*$ for
some constant $\lambda_0<1$.

For c), the global lower bound, property \eqref{e:err-red-pos-l}, and the
Galerkin orthogonality give
\begin{equation}\label{e:osc-red-pos-pf-1}
\begin{split}
 \alpha_1(1-2\mu)\|h_{P}^tr_P\|^2
&\leq
(1-2\mu)\left( \len u-u_{P}\ren^2 + \osc(u_P,P)^2\right)\\
&\leq
\len u-u_{P}\ren^2 + \osc(u_P,P)^2 - 2\len u-u_{P'}\ren^2 - 2\osc(u_{P'},P')^2\\
&\leq
\len u_{P}-u_{P'}\ren^2 + \osc(u_P,P)^2 - 2\,\osc(u_{P'},P')^2.
\end{split}
\end{equation}
For $\tau\in\hat{P}\cap\hat{P}'$ with all its neighbors also in
$\hat{P}\cap\hat{P}'$, 
where $\hat{P}'$ is the uniform refinement of $P'$,
we have $(Q_{\hat P}v)|_\tau = (Q_{\hat P'}v)|_\tau$, and also $h_P=h_{P'}$ there.
Hence, with $\hat{\omega}=\bigcup\{\tau\in\hat{P}:\exists\sigma\in P\setminus
P',\,\overline{\tau}\cap\overline{\sigma}\neq\varnothing\}$, we have
\begin{multline}
\|h_{P}^{t}r_P - h_{P}^{-t}Q_{\hat P} h_{P}^{2t}r_P\|_{\Gamma\setminus\hat{\omega}}
\leq
\|h_{P'}^{t}r_P - h_{P'}^{-t}Q_{\hat P'} h_{P'}^{2t}r_P\|\\
\leq
\|h_{P'}^{t}r_{P'} - h_{P'}^{-t}Q_{\hat P'}
h_{P'}^{2t}r_{P'}\|
+
\|h_{P'}^{t}A(u_P-u_{P'}) - h_{P'}^{-t}Q_{\hat P}
h_{P'}^{2t}A(u_P-u_{P'})\|\\
\leq
\osc(u_{P'},P') + C_Q^{1/2}\|h_{P'}^{t}A(u_P-u_{P'})\|.
\end{multline}
Combining this with \eqref{e:est-dom-osc}, we infer
\begin{equation}
\osc(u_P,P)^2
\leq
2\,\osc(u_{P'},P')^2
+
C_Q\|h_{P}^{t}r_P\|_{\omega^*}^2
+
2C_QC_A\len u_P-u_{P'}\ren^2,
\end{equation}
which, then on account of \eqref{e:osc-red-pos-pf-1} and the local
discrete upper bound, proves the claim.
\end{proof}

Now we prove an analogue of Proposition \ref{p:cont-zero} on error reduction.

\begin{proposition}\label{p:cont-pos}
Let the assumption \eqref{e:inverse-ineq-pos} of Lemma \ref{l:osc-red-pos} hold.
Let $P,P'\in\adm(P_0)$ be with $P\preceq P'$,
and let
$\Gamma^*=\intr\bigcup_{\tau\in P\setminus P'}\overline\tau$. 
Suppose, for some $\vartheta\in(0,1)$ that
\begin{equation}\label{e:dorfler-pos}
\|h_P^tr_P\|_{\Gamma^*} \geq\vartheta\|h_P^tr_P\|.
\end{equation}
Then there exist constants $\gamma\geq0$ and $\mu\in(0,1)$ such that
\begin{equation}
\len u-u_{P'}\ren^2+\gamma \|h_{P'}^tr_{P'}\|^2 \leq\mu\left(\len u-u_{P}\ren^2+\gamma \|h_P^tr_P\|^2\right).
\end{equation}
\end{proposition}

\begin{proof}
From the Galerkin orthogonality and Lemma \ref{l:osc-red-pos} b) we have
\begin{multline}
\len u-u_{P'}\ren^2 + \gamma \|h_{P'}^tr_{P'}\|^2
=
\len u-u_{P}\ren^2 - \len u_{P}-u_{P'}\ren^2 + \gamma \|h_{P'}^tr_{P'}\|^2\\
\leq
\len u-u_{P}\ren^2
+ \gamma(1+\delta)\|h_{P}^tr_{P}\|^2 
- \gamma\lambda(1+\delta) \|h_{P}^tr_{P}\|_{\Gamma^*}^2 
+ (\gamma C_\delta -1) \len u_P-u_{P'} \ren^2\\
\leq
\len u-u_{P}\ren^2
+ \gamma(1+\delta)\|h_{P}^tr_{P}\|^2 
- \gamma\lambda(1+\delta) \|h_{P}^tr_{P}\|_{\Gamma^*}^2,
\end{multline}
for $0<\gamma\leq 1/C_\delta$.
The idea, introduced in \citet*{CKNS08}, is to use the global upper bound and
the D\"orfler property \eqref{e:dorfler-pos} to bound fractions of the
first two terms by the third term.
For any $a,b>0$ we have
\begin{multline}
\len u-u_{P'}\ren^2 + \gamma \|h_{P'}^tr_{P'}\|^2\\
\leq
(1-\gamma a)\len u-u_{P}\ren^2
+ \gamma(Ca+1+\delta)\|h_{P}^tr_{P}\|^2 
- \gamma\lambda(1+\delta) \|h_{P}^tr_{P}\|_{\Gamma^*}^2\\
\leq
(1-\gamma a)\len u-u_{P}\ren^2
+ \gamma(Ca+1+\delta-b)\|h_{P}^tr_{P}\|^2\\ 
+ \gamma\left(\frac{b}{\vartheta^2}-\lambda(1+\delta)\right) \|h_{P}^tr_{P}\|_{\Gamma^*}^2,
\end{multline}
with the constant $C$ coming from the global upper bound.
We choose $b=\lambda\vartheta^2>0$ so that the third term is negative no matter
how small $\delta>0$ is,
and then choose $a>0$ and $\delta>0$ so small that $Ca+\delta<b$.
\end{proof}

For $u\in\tilde{H}^t(\Gamma)$ the solution of $Au=f$ with $f\in L^2(\Gamma)$,
and for $P\in\adm(P_0)$,
we define
\begin{equation}\label{e:dist-r-pos}
\dist_0(u,S_P) = \inf_{v\in S_P} \left( \len u-v \ren^2+\osc(v,P)^2 \right)^{\frac12}.
\end{equation}
Furthermore, for $\eps>0$ we define 
\begin{equation}
\card_0(u,\eps) = \min\{\#P-\#P_0:P\in\adm(P_0),\,\dist_0(u,S_P)\leq\eps\}.
\end{equation}
The following is an analogue of Proposition \ref{p:card-zero},
and we omit the proof, since the proof of Proposition \ref{p:card-zero} can be applied here {\em mutatis mutandis}.

\begin{proposition}\label{p:card-pos}
Let the assumption \eqref{e:inverse-ineq-pos} of Lemma \ref{l:osc-red-pos} hold.
Let $P\in\adm(P_0)$,
and let $\theta\in(0,\theta^*)$ with
$\theta^*=\frac{\alpha_1}{C_Q+\beta_1(1+2C_AC_Q)}$.
Suppose that $R\subseteq P$ is a subset whose cardinality is minimal up to a constant factor,
among all $R\subseteq P$ satisfying
\begin{equation}
\|h_P^tr_P\|_{\Gamma^*(R)}^2 \geq\theta\|h_P^tr_P\|^2,
\end{equation}
where $\Gamma^*(R)=\intr\bigcup_{\tau\in R}\overline\tau$.
Then we have
\begin{equation}
\#R\lesssim\card_0(u,\eps),
\end{equation}
where $\eps$ is defined by
\begin{equation}\label{e:def-eps-pos}
\eps^2 = \frac{\theta^*-{\theta}}{2C_G \theta^*} \left( \len u-u_{P}\ren^2+\osc(P,\Gamma)^2 \right).
\end{equation}
\end{proposition}

For completeness, in the rest of this section
we give an explicit pseudocode for the adaptive BEM, 
then define the relevant approximation classes,
and finally record a theorem on quasi-optimality.

\vspace{2mm}
\begin{algorithm}[H]
\SetKwInOut{Params}{parameters}
\SetKwInOut{Output}{output}
\SetKwFor{For}{for}{do}{endfor}
\Params{conforming partition $P_0$, and $\theta\in[0,1]$}
\Output{$P_k\in\adm(P_0)$ and $u_k\in S_{P_k}$ for all $k\in\N_0$}
\BlankLine
\For{$k=0,1,\ldots$}{
Compute $u_k\in S_{P_k}$ as the Galerkin approximation of $u$ from
$S_{P_k}$\;\nllabel{a:galsolve-pos}
Identify a minimal (up to a constant factor) set $R_k\subset P_k$ of triangles satisfying
\begin{equation}
\|h_{P_k}^tr_k\|_{\Gamma^*} \geq\theta\|h_{P_k}^tr_k\|,
\end{equation}
where $r_k=f-Au_k$ and $\Gamma^*=\intr\bigcup_{\tau\in R_k}\overline\tau$\;\nllabel{a:mark-pos}
Set $P_{k+1}=\refine(P_k,R_k)$\;\nllabel{a:make-conf-pos}
}
\caption{Adaptive BEM ($t>0$)}\label{a:abem-pos}
\end{algorithm}
\vspace{2mm}

We define the \emph{approximation class} $\mathcal{A}_{0,s}\subset
A^{-1}(L^2(\Gamma))$ with $s\geq0$ to be the set of $u$ for which
\begin{equation}
|u|_{\mathcal{A}_{0,s}} = \sup_{\eps>0} \left( \card_0(u,\eps)^s\eps \right)<\infty.
\end{equation}

The following result is immediate.

\begin{theorem}\label{t:comp-pos}
Let the assumption \eqref{e:inverse-ineq-pos} of Lemma \ref{l:osc-red-pos} hold,
and in Algorithm \ref{a:abem-pos}, suppose that $\theta\in(0,\theta^*)$ with $\theta^*=\frac{\alpha_1}{C_Q+\beta_1(1+2C_AC_Q)}$.
Let $f\in L^2(\Gamma)$ and $u\in\mathcal{A}_{0,s}$ for some $s>0$.
Then we have
\begin{equation}
\len u-u_k\ren^2+\osc(u_k,P_k)^2
\leq
C |u|_{\mathcal{A}_{0,s}}^2 (\#P_k-\#P_0)^{-2s},
\end{equation}
where $C>0$ is a constant.
\end{theorem}

\section{Negative order operators}
\label{s:negative}

Finally, we turn to the case $t<0$.
The domain $\Gamma$ can be either a closed manifold or a connected
polygonal subset of a closed manifold.
Recall that we are dealing with the linear homeomorphism $A:\tilde{H}^{t}(\Gamma)\to H^{-t}(\Gamma)$.
We will use the piecewise constant finite element spaces $S_P^0$,
and if $n\geq2$, take conforming triangulations as the class of admissible partitions.
Recall that $u_P\in S_P^0$ is the Galerkin approximation of $u$ from $S_P^0$,
and that $r_P=f-Au_P$ is its residual.

\cite{Faer00,Faer02} established the equivalence
\begin{equation}\label{e:faermann}
\len u-u_P\ren^2 \lesssim \sum_{z\in N_P} |r_P|_{-t,\omega(z)}^2
\lesssim \len u-u_P\ren^2,
\end{equation}
where $N_P$ is the set of all vertices in the triangulation $P$,
and proposed to use the local quantities $|r_P|_{-t,\omega(z)}$ as
error indicators for adaptive refinements.
An alternative proof, using interpolation spaces appeared in \citet*{CMS01}.
However, the discrete local counterparts to this equivalence have been open.
On the other hand, the weighted residual type error indicators,
$h_\tau^{1+t}|r_P|_{1,\tau}$ for $\tau\in P$,
have been around for a while, with a guaranteed global upper bound
\begin{equation}\label{e:global-upper-cms}
\len u-u_P\ren^2 \lesssim \sum_{\tau\in P} h_\tau^{2(1+t)}|r_P|_{1,\tau}^2,
\end{equation}
cf. \cite{CMS01}.
These indicators are computationally more attractive,
but for locally refined meshes no global lower bound or discrete local estimates have been known,
until the appearance of \citet*{FKMP11,FKMP11a} and this work.
The following lemma establishes the missing bounds for the
afore-mentioned error indicators.

\begin{lemma}\label{l:res-neg}
Let $P,P'\in\adm(P_0)$ be with $P\preceq P'$,
and let
$\Gamma^*=\intr\bigcup_{\tau\in P\setminus P'}\overline{\tau}$.
Furthermore, with $\omega^*=\omega(\Gamma^*)$ and $\phi_z\in S^1_P$ the standard nodal basis function at $z\in N_P$, let 
$\phi=\sum_{z\in N_P\cap\,\omega^*}\phi_z$.
Then for any $v\in S_{P'}^1$, it holds that
\begin{equation}\label{e:local-discrete-up-faermann}
  \len u_{P}-u_{P'} \ren^2
  \lesssim
  \sum_{z\in N_P\cap\,\omega^*} |r_{P}|_{-t,\omega(z)}^2
  + \|h_{P'}^{t}(r_{P}-v)\|_{\Gamma^*}^2 
  + \|\phi\,r_{P}-v\|_{H^{-t}(\Gamma)}^2.
\end{equation}
In addition to the above, for some $r\in[-t,\frac32)\cap(0,\frac12-2t)\cap[0,\mns]$, let $f\in H^r(\Gamma)$ and let $A:\tilde{H}^{r+2t}(\Gamma)\to H^{r}(\Gamma)$ be bounded.
Then for any $v\in S_{\hat P}^1$ and $w\in S_{P'}^0$, where $\hat P$ is the uniform refinement of $P$, we have
\begin{multline}\label{e:local-discrete-low-faermann}
  \sum_{z\in N_P\cap\,\omega^*} |r_{P}|_{-t,\omega(z)}^2
  \lesssim
  \sum_{z\in N_P\cap\,\omega^*} h_z^{2(r+t)}|r_{P}|_{r,\omega(z)}^2
  \lesssim
  \len u_{P}-u_{P'} \ren^2\\
  + \|h_{P}^{t}(r_{P}-w)\|_{\omega^*}^2 
  + \|h_{P}^{t}(r_{P}-v)\|_{\omega^*}^2 
  + \sum_{z\in N_P\cap\,\omega^*} h_z^{2(r+t)} |r_{P}-v|_{r,\omega(z)}^2,
\end{multline}
where $h_z=\min_{\{\tau:z\in\overline{\tau}\}}h_\tau$.
Moreover, for any $v\in S_{\hat P}^1$, we have the global bounds
\begin{multline}\label{e:global-neg}
  \sum_{z\in N_P} |r_{P}|_{-t,\omega(z)}^2
  \lesssim
  \sum_{z\in N_P} h_z^{2(r+t)}|r_{P}|_{r,\omega(z)}^2
  \lesssim
  \len u-u_{P} \ren^2\\
  + \|h_{P}^{t}(r_{P}-v)\|^2 
  + \sum_{z\in N_P} h_z^{2(r+t)} |r_{P}-v|_{r,\omega(z)}^2.
\end{multline}
\end{lemma}

\begin{proof}
Set $s=-t>0$.
Let $e=u_{P'}-u_{P}$, and let $e_P\in S_P^0$ be the $L^2$-orthogonal projection of $e$ onto $S_P^0$.
Then for any $v\in S^1_{P'}$, we have
\begin{equation}
\begin{split}
\langle Ae,e\rangle
&= \langle r_P,e\rangle
= \langle r_P,e-e_P\rangle
= \langle r_P,e\rangle_{\Gamma^*}
= \langle r_P-v,e\rangle_{\Gamma^*}+\langle v,e\rangle_{\Gamma^*}\\
&\leq \|h_{P'}^{-s}(r_P-v)\|_{\Gamma^*} \|h_{P'}^{s}e\|_{\Gamma^*}+\|v\|_{H^{s}}\|e\|_{\tilde{H}^{-s}},
\end{split}
\end{equation}
where we have used the fact that $e-e_P=0$ outside $\Gamma^*$.
Upon using the inverse inequality $\|h_{P'}^{s}e\|\lesssim\|e\|_{\tilde{H}^{-s}}$, this gives
\begin{equation}
\len e\ren
\lesssim \|h_{P'}^{-s}(r_P-v)\|_{\Gamma^*}+\|v\|_{H^{s}}
\lesssim \|h_{P'}^{-s}(r_P-v)\|_{\Gamma^*} + \|\phi\,r_P-v\|_{H^{s}} + \|\phi\,r_P\|_{H^{s}}.
\end{equation}
To localize the last term, we follow an approach from \cite{CMS01}.
With $\phi_z\in S^1_P$ the standard nodal basis function at $z\in N_P$,
we have
\begin{equation}
\|\phi\,r_P\|_{H^{s}}
\lesssim
\sum_{z\in N_P\cap\,\omega^*}\|\phi_zr_P\|_{H^{s}(\omega(z))}.
\end{equation}
The linear operator $T:H^s(\omega(z))\to H^s(\omega(z))$ defined by 
$Tf=(f-\langle f,1\rangle_{\omega(z)})\phi_z$ satisfies $\|Tf\|_{H^s}\lesssim|f|_{H^s(\omega(z))}$, cf. \cite{CMS01}.
Hence
\begin{equation}
\|\phi_zr_P\|_{H^{s}}
=
\|\phi_zr_P-\phi_z\langle r_P,1\rangle_{\omega(z)}\|_{H^{s}}
\lesssim
\|r_P\|_{H^{s}(\omega(z))},
\end{equation}
where we have taken into account that $r_P\perp S_P^0$,
and \eqref{e:local-discrete-up-faermann} follows.

For the lower bounds, we first prove a couple of inequalities involving the auxiliary error indicator $\|h_P^{-s}r_P\|_{\tau}$.
Let $w\in S_{P'}^0$, and let $w_P\in S_P^0$ be the $L^2$-orthogonal projection of $h_P^{-2s}w$ onto $S_P^0$.
Then we have
\begin{equation}
\begin{split}
\langle h_P^{-s}w,h_P^{-s}w\rangle
&= \langle w-r_P,h_P^{-2s}w\rangle + \langle A(u_{P'}-u_P),h_P^{-2s}w-w_P\rangle\\
&\leq \langle w-r_P,h_P^{-2s}w\rangle + \|A(u_{P'}-u_P)\|_{H^s(\Gamma^*)}\|h_P^{-2s}w-w_P\|_{\tilde{H}^{-s}(\Gamma^*)}\\
&\lesssim \|h_P^{-s}(w-r_P)\|_{\supp\,w}\|h_P^{-s}w\| + \|u_{P'}-u_P\|_{H^{-s}(\Gamma)}\|h_P^{-s}w\|_{\Gamma^*},
\end{split}
\end{equation}
where in the last step we used \eqref{e:direct-neg}, and so
\begin{equation}\label{e:res-neg-1}
\|h_P^{-s}r_P\|_{\supp\,w}\lesssim \len u_{P}-u_{P'}\ren + \|h_P^{-s}(r_P-w)\|_{\supp\,w}.
\end{equation}
Let $v\in L^2$, and let $v_P\in S_P^0$ be the $L^2$-orthogonal projection of $v$ onto $S_P^0$.
Using \eqref{e:direct-neg}, then we have
\begin{equation}
\begin{split}
\langle h^{-s}r_P,h^{s}v\rangle
&= \langle r_P,v\rangle
= \langle A(u-u_P),v-v_P\rangle\\
&\lesssim \|A(u-u_P)\|_{H^s(\Gamma)}\|v-v_P\|_{\tilde{H}^{-s}(\Gamma)}
\lesssim \len u-u_P \ren \|h^{s}v\|,
\end{split}
\end{equation}
establishing
\begin{equation}\label{e:res-neg-2}
\|h_P^{-s}r_P\|\lesssim \len u-u_{P}\ren.
\end{equation}

On the other hand, for any $N\subseteq N_P$ and $v\in S_{\hat P}^1$, we have
\begin{equation}
\begin{split}
\sum_{z\in N} h_z^{2(r-s)} |r_P|_{r,\omega(z)}^2
&\lesssim
\sum_{z\in N} \left( h_z^{2(r-s)} |v|_{r,\omega(z)}^2 + h_z^{2(r-s)} |r_P-v|_{r,\omega(z)}^2\right)\\
&\lesssim
\sum_{z\in N} \left( \|h_P^{-s}v\|_{\omega(z)}^2 + h_z^{2(r-s)} |r_P-v|_{r,\omega(z)}^2 \right)\\
&\lesssim
\|h_P^{-s}v\|_{\gamma}^2 + \sum_{z\in N} h_z^{2(r-s)} |r_P-v|_{r,\omega(z)}^2\\
&\lesssim
\|h_P^{-s}r_P\|_{\gamma}^2 + \|h_P^{-s}(v-r_P)\|_{\gamma}^2 + \sum_{z\in N} h_z^{2(r-s)} |r_P-v|_{r,\omega(z)}^2,
\end{split}
\end{equation}
where $\gamma=\intr\bigcup_{z\in N}\omega(z)$.
Then the second inequalities in \eqref{e:local-discrete-low-faermann} and \eqref{e:global-neg} follow from \eqref{e:res-neg-1} 
and \eqref{e:res-neg-2}, respectively.
Finally, the first inequalities in \eqref{e:local-discrete-low-faermann} and \eqref{e:global-neg} are a
consequence of the fact that $r_P$ is $L^2$-orthogonal to $S_P^0$.
\end{proof}

Let us record some useful bounds on the various oscillation terms
that appeared in the preceding lemma.

\begin{lemma}\label{l:osc-bnd-neg}
Let $P,P'\in\adm(P_0)$ be with $P\preceq P'$,
and let $\gamma=\intr\bigcup_{\tau\in Q}\overline{\tau}$ with some
$Q\subseteq P$.
Assume that $f\in H^{r}(\Gamma)$ and that $A:\tilde{H}^{r+2t}(\Gamma)\to H^{r}(\Gamma)$ is bounded for some $r\in(-t,\frac32)\cap(0,\frac12-2t)\cap[0,\mns]$, so that $r_P\in
H^{r}(\Gamma)$.
Then we have
\begin{equation}\label{e:w-osc}
  \min_{w\in S_{P'}^0}\|h_{P}^{t}(r_{P}-w)\|_{\gamma}^2
  \lesssim
  \sum_{\tau\in Q} h_\tau^{2(r+t)} |r_P|_{r,\tau}^2.
\end{equation}
With $v=Q_{\hat P}r_P\in S_{\hat P}^1$ the quasi-interpolant of $r_P$, we also have
\begin{equation}\label{e:up-osc}
  \|h_{P}^{t}(r_{P}-v)\|_{\gamma}^2 
  + \sum_{z\in N_P\cap\gamma} |r_{P}-v|_{-t,\omega(z)}^2
  \lesssim
  \sum_{z\in N_P\cap\overline{\gamma}} h_z^{2(r+t)} |r_P|_{r,\omega(z)}^2,
\end{equation}
where $h_z=\min_{\{\tau:z\in\overline{\tau}\}}h_\tau$.
Finally, with $\Gamma^*$ and $\omega^*$ as in Lemma \ref{l:res-neg},
there exists $v\in S^1_{P'}$ such that $v=0$ outside $\omega^*$ and that
\begin{equation}\label{e:low-osc}
  \|h_{P'}^{t}(r_{P}-v)\|_{\Gamma^*}^2 
  + \|\phi\,r_{P}-v\|_{H^{-t}(\Gamma)}^2
  \lesssim
  \sum_{z\in N_P\cap\,\omega^*} h_z^{2(r+t)} |r_P|_{r,\omega(z)}^2.
\end{equation}
\end{lemma}

\begin{proof}
The estimate \eqref{e:w-osc} is a standard direct
estimate, and \eqref{e:up-osc} follows from
\begin{equation}
  \sum_{z\in N_P\cap\gamma} \|r_{P}-v\|_{-t,\omega(z)}^2
  \lesssim
  \sum_{z\in N_P\cap\gamma} h_z^{2(r+t)} |r_P|_{r,\omega(\omega((z))}^2
  \lesssim
  \sum_{z\in N_P\cap\overline{\gamma}} h_z^{2(r+t)} |r_P|_{r,\omega(z)}^2,
\end{equation}
where we have used \eqref{e:direct-quasi}, \eqref{e:slobodeckij-faermann}, and the local finiteness of
the mesh.

For \eqref{e:low-osc}, let $v\in S^1_{P'}$ be defined by
\begin{equation}
 v = \sum_{z\in N_{P'}\cap\,\omega^*} \langle r_P,1\rangle_{\omega'(z)}\phi_z,
\end{equation}
where $\phi_z\in S^1_{P'}$ is the standard nodal basis function at $z\in N_{P'}$,
and $\langle \cdot,\cdot\rangle_{\omega'(z)}$ is the $L^2$-inner product on the star $\omega'(z)$ around $z$ with respect to $P'$.
Then we have
\begin{equation}
  \|h_{P'}^{t}(r_{P}-v)\|_{\Gamma^*}^2 
  \lesssim
  \sum_{z\in N_{P'}\cap\,\omega^*} |r_P|_{-t,\omega'(z)}^2
  \lesssim
  \sum_{z\in N_{P}\cap\,\omega^*} |r_P|_{-t,\omega(z)}^2,
\end{equation}
and from $r_P\perp_{L^2}S_P^0$ we infer the bound \eqref{e:low-osc} for its first term.
For the second term, we have
\begin{equation}
 \phi\,r_{P} - v = \sum_{z\in N_{P'}\cap\,\omega^*} \left(r_P-\langle r_P,1\rangle_{\omega'(z)}\right)\phi_z,
\end{equation}
and hence
\begin{equation}
 \|\phi\,r_{P} - v\|_{H^{-t}(\Gamma)}^2 
 \lesssim 
 \sum_{z\in N_{P'}\cap\,\omega^*} \left\|\left(r_P-\langle r_P,1\rangle_{\omega'(z)}\right)\phi_z\right\|_{H^{-t}(\omega'(z))}^2.
\end{equation}
As in the proof of Lemma \ref{l:res-neg},
now by using the boundedness of $f\mapsto\left(f-\langle f,1\rangle_{\omega'(z)}\right)\phi_z$ in $H^{-t}(\omega'(z))$,
and then by employing the orthogonality $r_P\perp_{L^2}S_P^0$ again, we establish the proof.
\end{proof}

\begin{remark}[Saturation assumption]
Similarly to Remark \ref{r:pos-sat}, one can also derive some results on the saturation assumption for the non-residual estimators from \citet*{EFLFP09}.
We skip the details.
\end{remark}

Combining \eqref{e:local-discrete-up-faermann} and \eqref{e:low-osc},
for the Faermann indicators we get
\begin{equation}
  \len u_{P}-u_{P'} \ren^2
  \lesssim
  \sum_{z\in N_P\cap\,\omega^*} |r_{P}|_{-t,\omega(z)}^2 + \sum_{z\in N_P\cap\,\omega^*} h_z^{2(r+t)} |r_P|_{r,\omega(z)}^2,
\end{equation}
and by \eqref{e:local-discrete-low-faermann}, the entire right hand side of the preceding inequality is controlled by the last term alone.
This leads to generalized weighted residual indicators,
for which we already have the global bounds \eqref{e:global-neg}.
Obviously, the case $r=1$, treated in \cite{FKMP11,FKMP11a} is computationally more attractive,
but the remaining cases $r\in(-t,1)$ become important if for instance $f\in H^r(\Gamma)\setminus H^1(\Gamma)$.

In the following, we \emph{fix} some $r\in(-t,\frac32)\cap(0,\frac12-2t)\cap[0,\nu]$ and assume that $f\in H^r(\Gamma)$,
and that $A:\tilde{H}^{r+2t}(\Gamma)\to H^{r}(\Gamma)$ is bounded.
Then defining the \emph{error estimator}
\begin{equation}
\eta(v,P,\gamma) = \left(\sum_{z\in N_P\cap\,\overline{\gamma}} h_z^{2(r+t)}|f-Av|_{r,\omega(z)}^2\right)^{\frac12},
\end{equation}
for $\gamma\subseteq\Gamma$ and $v\in S_P^0$, with $\eta(P,\gamma)=\eta(u_P,P,\gamma)$,
in the context of Lemma \ref{l:res-neg} we have
\begin{equation}\label{e:local-discrete-est-neg}
  \len u_{P}-u_{P'}\ren^2\leq \beta_1 \eta(P,\Gamma^*)^2,
\end{equation}
where $\beta_1>0$ is a constant we will refer to later.
Let there be a map $z\mapsto\tau_P(z):N_P\to P$ for any $P\in\adm(P_0)$, satisfying $z\in\overline{\tau_P(z)}$.
For $v\in S_P^0$, let
\begin{equation}
\osc_r(v,P) = \left( \|h_{P}^{t}(f-Av-w)\|^2 + \sum_{z\in N_{\hat P}} h_z^{2(r+t)} |f-Av-w|_{r,\hat\omega(z)}^2 \right)^{\frac12},
\end{equation}
where $\hat\omega(z)$ is the star around $z\in N_{\hat P}$ with respect to the partition $\hat P$,
and $w=Q_{\hat P}(f-Av)\in S^1_{\hat P}$ is the quasi-interpolator of $f-Av$, defined by
\begin{equation}
 Q_{\hat P}g = \sum_{z\in N_{\hat P}} g_z(z)\phi_z,
\end{equation}
$\phi_z\in S^1_{\hat P}$ is the standard nodal basis function at $z$,
and $g_z\in\mathbb{P}_1$ is the $L^2$-orthogonal projection of $g$ onto the affine functions on $\tau_{\hat P}(z)$.
This quasi-interpolator, as a variation on the Cl\'ement interpolator, is introduced in \cite{Osw94}.
We remark that we will only need the idempotence $Q_{\hat P}^2=Q_{\hat P}$,
so for example the Scott-Zhang operator could have been employed instead.
Let $g'=g-Q_{\hat P}g$ with $g\in H^r(\Gamma)$, and let $y\in N_P$.
Then by idempotence, we have
\begin{equation}
|g'|_{r,\omega(y)}^2
=
|g'-Q_{\hat P}g'|_{r,\omega(y)}^2
\lesssim
\sum_{z\in N_{\hat P}\cap\,\overline{\omega(z)}} |(g'-g'_z(z))\phi_z|_{r,\hat\omega(z)}^2,
\end{equation}
which implies by the boundedness of $g'\mapsto (g'-g'_z(z))\phi_z$ in $H^r(\hat\omega(z))$, that
\begin{equation}
\sum_{y\in N_{P}} |g-Q_{\hat P}g|_{r,\omega(y)}^2
\lesssim
\sum_{z\in N_{\hat P}} |g-Q_{\hat P}g|_{r,\hat\omega(z)}^2.
\end{equation}
In particular, in the context of Lemma \ref{l:res-neg} we have
\begin{equation}\label{e:global-est-neg}
  \len u-u_{P}\ren^2
  \lesssim 
  \alpha_1 \eta(P,\Gamma)^2
  \leq
  \len u-u_{P}\ren^2 + \osc_r(u_P,P)^2,
\end{equation}
where $\alpha_1>0$ is a constant.
On the other hand, by stability of $Q_{\hat P}$ there exists a constant $C_Q>0$ such that 
\begin{equation}\label{e:est-dom-osc-neg}
\|h_{P}^{t}(r_P-Q_{\hat P}r_P)\|_{\hat\omega(\gamma)}^2 + \sum_{z\in N_{\hat P}\cap\,\overline{\hat\omega(\gamma)}} h_z^{2(r+t)} |r_P-Q_{\hat P}r_P|_{r,\hat\omega(z)}^2
\leq
C_Q\eta(P,\gamma)^2,
\end{equation}
for any $P\in\adm(P_0)$ and $\gamma\subseteq\Gamma$, with $\hat\omega(\gamma)=\bigcup\{\hat\omega(z):z\in N_{\hat P}\cap\overline{\gamma}\}$.

In the rest of this section, 
we follow the pattern of the preceding two sections,
and skip the proofs that closely resemble those
given in the previous sections.

\begin{lemma}\label{l:osc-red-neg}
Assume that
\begin{equation}\label{e:inverse-ineq-neg}
\sum_{z\in N_P}h_z^{2(r+t)}|Av|_{r,\omega(z)}^2
\leq
C_A \len v \ren^2,
\qquad
v\in S_{P}^0,
\end{equation}
for $P\in\adm(P_0)$, with the constant $C_A=C_A(A,\shape,\grade)$.
Then the followings hold.
\begin{enumerate}[a)]
\item
There is a constant $C_G>0$ such that
\begin{equation}\label{e:gal-opt-neg}
\len u-u_P \ren^2 + \osc_r(u_P,P)^2 \leq C_G \inf_{v\in S_P^0} \left( \len u-v \ren^2 + \osc_r(v,P)^2 \right),
\end{equation}
for any $P\in\adm(P_0)$.
\item
There exists a constant $\lambda>0$ such that
\begin{equation}\label{e:osc-red-neg}
\eta(P',\Gamma)^2 
\leq 
(1+\delta)\eta(P,\Gamma)^2 
- \lambda(1+\delta) \eta(P,\Gamma^*)^2 
+ C_\delta \len u_P-u_{P'} \ren^2,
\end{equation}
for any $P,P'\in\adm(P_0)$ with $P\preceq P'$,
and for any $\delta>0$, with $C_\delta$ depending on $\delta$,
where $\Gamma^*=\intr\bigcup_{\tau\in P\setminus P'}\overline\tau$.
\item
Let $P,P'\in\adm(P_0)$ be such that $P\preceq P'$,
and let
$\Gamma^*=\intr\bigcup_{\tau\in P\setminus P'}\overline\tau$.
If, for some $\mu\in(0,\frac12)$ it holds that
\begin{equation}\label{e:err-red-neg-l}
\len u-u_{P'}\ren^2+\osc_r(u_{P'},P')^2 \leq\mu\left(\len u-u_{P}\ren^2+\osc_r(u_{P},P)^2\right),
\end{equation}
then with $\theta=\frac{\alpha_1(1-2\mu)}{C_Q+\beta_1(1+2C_AC_Q)}$, we have
\begin{equation}\label{e:dorfler-neg-l}
\eta(P,\Gamma^*)^2 \geq \theta\, \eta(P,\Gamma)^2,
\end{equation}
where $\alpha_1>0$ and $\beta_1>0$ are the constant from \eqref{e:global-est-neg} and \eqref{e:local-discrete-est-neg},
respectively.
\end{enumerate}
\end{lemma}

\begin{remark}
The estimate \eqref{e:inverse-ineq-neg} is proved in Section
\ref{s:inverse} for a general class of singular integral operators of
negative order, under the hypothesis
that $\hat{A}:H^{t+\sigma}(\Omega)\to H^{-t+\sigma}(\Omega)$ is bounded for some
$\sigma>\frac{n}2$, and that $0<r+t<\sigma$, where $\hat{A}$
is an extension of $A$ to $\Omega$ if $\Gamma$ is open, and $\hat{A}=A$ if $\Gamma=\Omega$.
Recalling that $\Omega$ is a $C^{\mns-1,1}$-manifold,
the condition on $\sigma$ translates to $\mns>\frac{n}2-t$, see Remark
\ref{r:bdd-bdry-int} for details.
Therefore for $n=1$ and $n=2$ we need a $C^{1,1}$
curve or a surface, respectively, assuming that $\frac12\leq-t<1$.
Note that if $\Gamma$ is open, the boundary of $\Gamma$ is allowed to be Lipschitz.
\end{remark}

\begin{proof}[Proof of Lemma \ref{l:osc-red-neg}]
The proofs of a) and c) go along the same lines as the corresponding proofs from the previous section, cf. Lemma \ref{l:osc-red-pos}.
In particular, the stability \eqref{e:est-dom-osc-neg} and the locality of $Q_{\hat P}$ are important. 

Claim b) is established if we show that
\begin{equation}\label{e:osc-red-neg-1}
\sum_{z\in N_{P'}\cap\,\overline{\Gamma^*}} h_z^{2(r+t)}|g|_{r,\omega'(z)}^2
\leq
\lambda
\sum_{y\in N_P\cap\,\overline{\Gamma^*}} h_y^{2(r+t)}|g|_{r,\omega(y)}^2,
\qquad
g\in H^r(\Gamma),
\end{equation}
with some constant $\lambda<1$,
where $\omega'(z)$ is the star around $z\in N_{P'}$ with respect to the partition $P'$.
For $\tau,\sigma\in P'$, let $I(\tau,\sigma)$ denote the interaction term between $\tau$ and $\sigma$
in the Slobodeckij (double integral) norm of $g$.
First of all, any diagonal term $I(\tau,\tau)$ that appears in the left hand side also appears in the right hand side,
and the corresponding factors satisfy $h_z\leq\lambda_0 h_y$ with some constant $\lambda_0<1$.
Henceforth we concentrate on the off-diagonal terms.
Note that by symmetry the order of $\tau$ and $\sigma$ is not important,
and that the number of occurrences of the particular (unordered) pair $(\tau,\sigma)$ in the left hand side of \eqref{e:osc-red-neg-1}
is equal to the number of $z\in N_{P'}\cap\,\overline{\Gamma^*}$ satisfying $\tau,\sigma\subset\omega'(z)$.
Suppose that the pair $(\tau,\sigma)$ appears in the left hand side exactly $\ell$ times, for $\ell\in[0,n-1]$.
Thus $\tau$ and $\sigma$ are contained in two triangles from $P$ that share a $k$-face for some $k\in[\ell-1,n]$,
where $k=n$ means that the two triangles coincide.
If this face is in $\overline{\Gamma^*}$, then the vertices of this face give at least $\ell$ points 
$y\in N_P\cap\,\overline{\Gamma^*}$ such that $\tau,\sigma\subset\omega(y)$,
meaning that the same pair appears in the right hand side at least $\ell$ times.
On the other hand, if the shared face is not in $\overline{\Gamma^*}$, 
this would mean that $\tau$ and $\sigma$ are triangles from $P$,
and they interact through only the vertices on $N_P\cap\partial\Gamma^*$.
We also see that the corresponding factors $h_z^{2(r+t)}$ shrink since $h_z$ is defined by taking minimum as $h_z=\min_{\{\tau\in P':z\in\overline{\tau}\}}h_\tau$,
proving the claim \eqref{e:osc-red-neg-1}.
\end{proof}

\begin{proposition}
Let the assumption \eqref{e:inverse-ineq-neg} of Lemma \ref{l:osc-red-neg} hold.
Let $P,P'\in\adm(P_0)$ be admissible partitions with $P\preceq P'$,
and let
$\Gamma^*=\intr\bigcup_{\tau\in P\setminus P'}\overline\tau$. 
Suppose, for some $\vartheta\in(0,1)$ that
\begin{equation}\label{e:dorfler-neg}
\eta(P,\Gamma^*) \geq \vartheta\,\eta(P,\Gamma).
\end{equation}
Then there exist constants $\gamma\geq0$ and $\mu\in(0,1)$ such that
\begin{equation}
\len u-u_{P'}\ren^2+\gamma\,\eta(P',\Gamma)^2 \leq\mu\left(\len u-u_{P}\ren^2+\gamma\,\eta(P,\Gamma)^2\right).
\end{equation}
\end{proposition}

Gearing towards a convergence rate analysis,
for $u\in\tilde{H}^t(\Gamma)$ the solution of $Au=f$ with $f\in H^r(\Gamma)$,
and for $P\in\adm(P_0)$,
we define
\begin{equation}\label{e:dist-r-neg}
\dist_r(u,S_P^0) = \inf_{v\in S_P^0} \left( \len u-v \ren^2+\osc_r(v,P)^2 \right)^{\frac12}.
\end{equation}
Furthermore, for $\eps>0$ we define 
\begin{equation}
\card_r(u,\eps) = \min\{\#P-\#P_0:P\in\adm(P_0),\,\dist_r(u,S_P^0)\leq\eps\}.
\end{equation}
We have the following result on the D\"orfler marking,
whose proof is entirely analogous to the proof of Proposition \ref{p:card-pos}.

\begin{proposition}\label{p:card-neg}
Let the assumption \eqref{e:inverse-ineq-neg} of Lemma \ref{l:osc-red-neg} hold.
Let $P\in\adm(P_0)$,
and let $\theta\in(0,\theta^*)$ with
$\theta^*=\frac{\alpha_1}{1+\beta_1(1+2C_A)}$.
Suppose that $R\subseteq P$ is a subset whose cardinality is minimal up to a constant factor,
among all $R\subseteq P$ satisfying
\begin{equation}
\eta(P,\Gamma^*(R))^2 \geq \theta\,\eta(P,\Gamma)^2.
\end{equation}
where $\Gamma^*(R)=\intr\bigcup_{\tau\in R}\overline\tau$.
Then we have
\begin{equation}
\#R\lesssim\card_r(u,\eps),
\end{equation}
where $\eps$ is defined by
\begin{equation}\label{e:def-eps-neg}
\eps^2 = \frac{\theta^*-{\theta}}{2C_G \theta^*} \left( \len u-u_{P}\ren^2+\osc_r(P,\Gamma)^2 \right).
\end{equation}
\end{proposition}

Now let us specify our adaptive algorithm.

\vspace{2mm}
\begin{algorithm}[H]
\SetKwInOut{Params}{parameters}
\SetKwInOut{Output}{output}
\SetKwFor{For}{for}{do}{endfor}
\Params{conforming partition $P_0$, and $\theta\in[0,1]$}
\Output{$P_k\in\adm(P_0)$ and $u_k\in S_{P_k}^0$ for all $k\in\N_0$}
\BlankLine
\For{$k=0,1,\ldots$}{
Compute $u_k\in S_{P_k}^0$ as the Galerkin approximation of $u$ from
$S_{P_k}^0$\;\nllabel{a:galsolve-neg}
Identify a minimal (up to a constant factor) set $R_k\subset P_k$ of triangles satisfying
\begin{equation}
\eta(P_k,\Gamma^*)^2 \geq \theta\,\eta(P_k,\Gamma)^2,
\end{equation}
where $r_k=f-Au_k$ and $\Gamma^*=\intr\bigcup_{\tau\in R_k}\overline\tau$\;\nllabel{a:mark-neg}
Set $P_{k+1}=\refine(P_k,R_k)$\;\nllabel{a:make-conf-neg}
}
\caption{Adaptive BEM ($t<0$)}\label{a:abem-neg}
\end{algorithm}
\vspace{2mm}

Finally, we introduce the \emph{approximation class} $\mathcal{A}_{r,s}\subset
A^{-1}(H^r(\Gamma))$ with $s\geq0$ to be the set of $u$ for which
\begin{equation}
|u|_{\mathcal{A}_{r,s}} = \sup_{\eps>0} \left( \card_r(u,\eps)^s\eps \right)<\infty,
\end{equation}
and record that our adaptive BEM produces optimally converging approximations.

\begin{theorem}\label{t:comp-neg}
Let the assumption \eqref{e:inverse-ineq-neg} of Lemma \ref{l:osc-red-neg} hold,
and in Algorithm \ref{a:abem-neg}, suppose that $\theta\in(0,\theta^*)$ with $\theta^*=\frac{\alpha_1}{1+\beta_1(1+2C_A)}$.
Let $f\in H^r(\Gamma)$ and $u\in\mathcal{A}_{r,s}$ for some $s>0$.
Then we have
\begin{equation}
\len u-u_k\ren^2+\osc_r(P_k,\Gamma)^2
\leq
C |u|_{\mathcal{A}_{r,s}}^2 (\#P_k-\#P_0)^{-2s},
\end{equation}
where $C>0$ is a constant.
\end{theorem}

\section{Inverse-type inequalities}
\label{s:inverse}

In this section we shall justify the inverse-type inequality
\eqref{e:inverse-ineq-gen},
which has been used in Lemmata \ref{l:osc-red-zero},
\ref{l:osc-red-pos}, and \ref{l:osc-red-neg},
and hence played a crucial role in our analysis.
We allow a general class of singular integral operators,
specified by the assumptions that follow.

We keep the assumptions formulated in Section \ref{s:general} still in
force.
We will be concerned only with closed manifolds (i.e., $\Gamma=\Omega$), since the
case of open surfaces $\Gamma\subset\Omega$ would follow by
restriction.
In addition, we assume that $\Omega$ is embedded in some Euclidean
space $\R^{N}$, so that the Euclidean distance function
$\dist:\Omega\times\Omega\to[0,\infty)$ is well defined.
Instead of the operator $A$ that featured in the previous sections,
we will consider in this section a more general \emph{bounded linear} operator $T:H^t(\Omega)\to
H^{-t}(\Omega)$,
hence removing the self-adjointness and coerciveness assumptions.
With $\Delta=\{(x,x):x\in\Omega\}$ the diagonal of $\Omega\times\Omega$,
we assume that there is a kernel $K\in L^1_{\mathrm{loc}}(\Omega\times\Omega\setminus\Delta)$
associated to $T$, meaning that
\begin{equation}
\langle Tu,v\rangle = \langle K,u\otimes v\rangle,
\end{equation}
whenever $u,v\in C^{\nu-1,1}(\Omega)$ have disjoint supports.
We assume that $K$ is smooth on
\begin{equation}
\Sigma=\{\tau\times\tau':\tau,\tau'\in P_0\}\setminus\Delta\subset\Omega\times\Omega\setminus\Delta,
\end{equation}
satisfying the estimate
\begin{equation}\label{e:standard-kernel}
|\partial_\xi^\alpha\partial_\eta^\beta K(\xi,\eta)| \leq
\frac{C_{\alpha,\beta}}{\dist(\xi,\eta)^{n+2t+|\alpha|+|\beta|}},
\qquad
(\xi,\eta)\in\Sigma,
\end{equation}
for all multi-indices $\alpha$ and $\beta$ satisfying
$n+2t+|\alpha|+|\beta|>0$.
Note that the partial derivatives are understood in local
coordinates (or, as we discussed in \S\ref{s:general}, in terms of the reference triangles).
The kernels satisfying this smoothness condition have been called \emph{standard kernels},
e.g., in \citet*{DHS06}.
Then one can show that the kernels of a wide range of boundary
integral operators are standard kernels, cf. \cite{Schn98}.

\begin{theorem}\label{t:inverse-ineq}
With $T^*$ denoting the adjoint of $T$, 
let both $T,T^*:H^{t+\sigma}(\Omega)\to H^{-t+\sigma}(\Omega)$ be bounded
for some $\sigma>\frac{n}2$.
Moreover, assume $s\geq0$ and $0<s+t<\sigma$.
Let $S^d_{P}\subset H^t(\Omega)$.
Then we have
\begin{equation}
\sum_{\tau\in P}h_\tau^{2(s+t)}\|Tv\|_{H^s(\tau)}^2 
\lesssim
\sum_{z\in N_P}h_z^{2(s+t)}\|Tv\|_{H^s(\omega(z))}^2 
\lesssim
\|v\|_{H^t(\Omega)}^2,
\qquad
v\in S^d_{P},
\end{equation}
for $P\in\adm(P_0)$,
where 
$N_P$ is the set of all vertices in the triangulation $P$,
$\omega(z)$ is the star around $z$ with respect to $P$,
and $h_z=\min_{\{\tau:z\in\overline{\tau}\}}h_\tau$.
\end{theorem}

As already mentioned, note that taking $v=0$ on a subcollection of triangles in the above theorem allows us to treat the case of open surfaces $\Gamma\subset\Omega$.

\begin{remark}\label{r:bdd-bdry-int}
The boundedness of both
$T,T^*:H^{t+\sigma}(\Omega)\to H^{-t+\sigma}(\Omega)$ has been proved for $\sigma<\frac12$ for general boundary integral operators on Lipschitz surfaces in \cite{Cost88},
and the endpoint case $\sigma=\frac12$ is established for boundary integral
operators associated to the Laplace operator on Lipschitz domains in
\cite{Verch84}.
Unfortunately, we see that the preceding results require more than
$\sigma=\frac12$.
Since $\mns\geq\sigma+|t|$,
we necessarily have $\mns>\frac{n}2+|t|$.
This allows Lipschitz curves for $|t|<\frac12$,
and $C^{1,1}$ curves and surfaces for $|t|<1$.
Even though in general it rules out polyhedral surfaces,
note that for the case of an open surface $\Gamma$, 
its boundary can be Lipschitz polygonal, as long as one can find a
smooth enough manifold $\Omega$ with $\Gamma\subset\Omega$.
As for the question of whether the boundedness holds for $\sigma<\nu-\frac12$ for $C^{\mns-1,1}$ domains,
let us note that the main ingredients of the results in \cite{Cost88} are the near-optimal trace theorem for Lipschitz domains, 
which appears in \cite{Cost88} and \cite{Ding96},
and a certain regularity result for the Poincar\'e-Steklov operator for Lipschitz domains, which appears, e.g., in \cite{McL00}.
The relevant version of the trace theorem has been proved for $C^{\mns-1,1}$ domains in \cite{Kim07}, see also \cite{Mar87}.
For the regularity of the Poincar\'e-Steklov operator, the author has not been able to locate in the literature a result strong enough to give the boundedness for $\sigma<\nu-\frac12$, although there are results, e.g., in \cite{McL00}, that imply $\sigma=\mns-1$ for standard boundary integral operators.
\end{remark}


The proof of Theorem \ref{t:inverse-ineq} is divided into several lemmas that follow.
The main idea is to decompose $Tv$ 
into the part that is in $S^d_P$, which we call the \emph{low frequency} part, 
and its complement, which we call the \emph{high frequency} part.
The low frequency part is readily handled by 
either the standard inverse estimates or
the new inverse estimates from \citet*{DFGHS04}.
To treat the high frequency part, 
we introduce a wavelet basis for the complement of $S^d_P$,
and as naturally suggested by the techniques we use,
the high frequency part is further decomposed into terms corresponding
to \emph{far-field}, \emph{near-field}, and \emph{local} interactions.
Please be warned that these names are only suggestive in that, e.g., the local interaction terms may contain 
interactions between two wavelets with non-overlapping supports, although they cannot be too far apart.

Our main analytic tool is a locally supported wavelet basis for the energy space $H^t$, 
with the dual multiresolution analysis based on piecewise polynomial-type spaces.
More specifically, we assume that there is a Riesz basis $\Psi=\{\psi_\lambda\}_{\lambda\in\nabla}$ of $H^t$ of {\em wavelet type},
whose dual, denoted by $\tilde\Psi=\{\tilde{\psi}_\lambda\}_{\lambda\in\nabla}$, is locally supported piecewise polynomial wavelets,
where $\nabla$ is a countable index set.
Now we expand on what we mean exactly by the various adjectives such
as ``wavelet type'' that characterize the bases $\Psi$ and
$\tilde{\Psi}$.

The collections $\Psi\subset H^t(\Omega)$ and
$\tilde\Psi\subset H^{-t}(\Omega)$ are biorthogonal: $\langle\psi_\lambda,\tilde\psi_\mu\rangle=\delta_{\lambda\mu}$,
and are \emph{Riesz bases} for their
corresponding spaces, meaning that
\begin{equation}\label{e:normeq}
\left\|\textstyle\sum_{\lambda\in\nabla}
  v_\lambda\psi_\lambda\right\|_{H^t(\Omega)} 
\eqsim
\|(v_\lambda)_\lambda\|_{\ell^2(\nabla)},
\qquad
\left\|\textstyle\sum_{\lambda\in\nabla}
  v_\lambda\tilde\psi_\lambda\right\|_{H^{-t}(\Omega)} 
\eqsim 
\|(v_\lambda)_\lambda\|_{\ell^2 (\nabla)},
\end{equation}
for any sequence $(v_\lambda)_\lambda\in\ell^2(\nabla)$.
Here the notation $X\eqsim Y$ means $Y\lesssim X\lesssim Y$.
Each wavelet $\psi_\lambda$ or $\tilde\psi_\lambda$ has a scale,
which is encoded by the function $|\cdot|:\nabla\to\N$.
We say that $\psi_\lambda$ and $\tilde\psi_\lambda$ have the scale
$2^{-|\lambda|}$, which is justified by the \emph{locality} properties
\begin{equation}
\diam(\supp\,\psi_\lambda)\lesssim2^{-|\lambda|},
\qquad
\diam(\supp\,\tilde\psi_\lambda)\lesssim2^{-|\lambda|}.
\end{equation}
We will also assume that the wavelet supports are \emph{locally finite},
in the sense that
\begin{equation}
\#\{\lambda:|\lambda|=\ell,\,B(x,2^{-\ell})\cap\supp\,\psi_\lambda\neq\varnothing\}\lesssim1,
\end{equation}
and similarly for the dual wavelets, where the bounds do not depend on
$\ell\in\N$ and $x\in\Omega$,
and $B(x,\rho)=\{y\in\Omega:\dist(x,y)<\rho\}$.
An immediate consequence of this property is that $\#\{\lambda\in\nabla:|\lambda|=\ell\}\lesssim2^{n\ell}$.

We assume that the wavelets have the so-called
{\em cancellation property of order} $p \in \N$, saying
that\footnote{Note that $p$ and $\tilde p$ are, respectively,
$\tilde d$ and ${d}$ as compared to, e.g., \cite{Stev04}.}
 there exists a constant $\eta>0$,
such that for any $q \in [1,\infty]$, for all continuous,
piecewise smooth functions $v$ on $P_0$ and $\lambda \in \nabla$,
\begin{equation}
|\langle v, \psi_{\lambda} \rangle| \lesssim
2^{-|\lambda|(\frac{n}{2}-\frac{n}{q}+t+p)}
\max_{\tau\in P_0}
|v|_{W^{p,q}(B(\supp\,\psi_{\lambda},2^{-|\lambda|}\eta) \cap \tau)},
\end{equation}
where for \(A \subset \R^{N}\) and \(\eps > 0\),
\(B(A,\eps):=\{y \in \R^{N}: {\rm dist}(A,y)< \eps\}\).

Furthermore, we assume that
for all $r \in [-p, \gamma)$, $s<\gamma$,
necessarily with  $|s|,|r| \leq \mns$,
\begin{equation}
\| w \|_{H^r(\Omega)} \lesssim 2^{\ell(r-s)}\|w\|_{H^s(\Omega)},
\qquad
\textrm{for}
\quad
w\in\spn\{\psi_{\lambda}:|\lambda|=\ell\},
\end{equation}
with $\gamma=\sup\{s:\Psi\subset H^{s}(\Omega)\}$,
and similarly for the dual wavelets, with $\gamma$ and $p$
replaced by $\tilde\gamma$ and $\tilde{p}$, respectively.

We assume that the norm equivalence
\begin{equation}\label{e:bnormeq}
\left\|\textstyle\sum_{\lambda\in\nabla}
  v_\lambda\tilde\psi_\lambda\right\|_{H^{s}(\Omega)}^2
\eqsim
\sum_{\lambda\in\nabla}2^{2(s+t)|\lambda|}|v_\lambda|^2,
\end{equation}
is valid for $s\in(-\tilde p,p)\cap(-\gamma,\tilde\gamma)$.
As far as the following proof of Theorem \ref{t:inverse-ineq} is concerned,
we will use only the``greater than'' part of the first norm equivalence in \eqref{e:normeq},
and the ``less than'' part of \eqref{e:bnormeq} with $s$ equal to the same parameter in the theorem.
For this and other reasons, in what follows we assume that the parameters $p$, $\tilde p$, $\gamma$, and $\tilde\gamma$ are sufficiently large.
Such a possibility is guaranteed by the constructions in \cite{DSch99}, see the remark below.

For $\lambda\in\nabla$, we define
$\supb_\lambda=\Omega\cap B_\lambda$, with $B_\lambda$ an open ball with $\diam(B_\lambda)\lesssim2^{-|\lambda|}$,
containing both $\supp\,\psi_\lambda$ and $\supp\,\tilde\psi_\lambda$.
Thus $\supb_\lambda$ can be thought of as a common support of
$\psi_\lambda$ and $\tilde\psi_\lambda$.
We then define
\begin{equation}
\Lambda = \{\lambda:\supb_\lambda\cap\tau\neq\varnothing\textrm{
  and }\diam(\supb_\lambda)\geq \delta h_\tau\textrm{ for some }\tau\in P\},
\end{equation}
where $\delta>0$ is a constant so small that
$\Omega_\lambda\cap\tau\neq\varnothing$ and $\lambda\not\in\Lambda$
imply $\Omega_\lambda\subset\omega(\tau)$ for $\tau\in P$.
Roughly speaking, the index set $\Lambda$ corresponds to the wavelets that are 
needed to resolve the finite element space $S^d_P$.
Note that the existence of such a $\delta>0$ is guaranteed by the shape regularity of $P$.
On the wavelet equivalent of $S^d_P$,
we assume the inverse inequality
\begin{equation}\label{e:wav-inverse-pos}
\|v\|_{H^s(\Omega)} \lesssim \|h_P^{-s}v\|,
\qquad
v\in S_\Lambda:=\spn\{\tilde\psi_\lambda:\lambda\in\Lambda\},
\end{equation}
for $s\in[0,\mns]$.
By a duality argument (test $v$ against $w\eqsim h_P^{2s}v$) 
this implies
\begin{equation}\label{e:wav-inverse}
\|h_P^sv\| \lesssim \|v\|_{H^{-s}(\Omega)},
\qquad
v\in S_\Lambda,
\end{equation}
for $s\in[0,\mns]$.

\begin{remark}
Concrete examples of wavelet bases satisfying all our assumptions are
given by the {\em duals} of the bases constructed in \cite{DSch99}.
On triangulations over a Lipschitz polyhedral surface, one can also use the construction in \cite{Stev03}.
The only reason for insisting on polynomial dual wavelets is that in
the proof of Lemma \ref{l:low-freq} below, we use the inverse estimate
\eqref{e:wav-inverse}, which is a consequence of \eqref{e:wav-inverse-pos}.
The latter estimate (hence both) can be proven by adapting the techniques from \cite{DFGHS04}.
\end{remark}

Now that we have settled on our main tool, we can start with {\em Proof of Theorem \ref{t:inverse-ineq}}.
Let $v\in S^d_P$ be as in the theorem.
Then first we estimate the part of $Tv$ that is in $S_\Lambda$.
We define the projection operator $Q_\Lambda:H^{-t}(\Omega)\to
S_\Lambda$ by
\begin{equation}
Q_\Lambda \sum_{\lambda\in\nabla} w_\lambda\tilde{\psi}_\lambda =
\sum_{\lambda\in\Lambda} w_\lambda\tilde{\psi}_\lambda.
\end{equation}

In what follows, 
we will abbreviate the Sobolev norms as $\|\cdot\|_{s,\omega}=\|\cdot\|_{H^s(\omega)}$ and $\|\cdot\|_{s}=\|\cdot\|_{H^s(\Omega)}$.

\begin{lemma}[Low frequency]\label{l:low-freq}
Let $s\geq0$, and
suppose that either $t>0$ or $s+t>0$.
Then we have
\begin{equation}
\sum_{z\in N_P}h_z^{2(s+t)}\|Q_{\Lambda}Tv\|_{s,\omega(z)}^2 \lesssim\|v\|_{t}^2
\qquad\textrm{for}\quad
v\in H^t(\Omega).
\end{equation}
\end{lemma}

\begin{proof}
First let $t\leq0$, and therefore $s+t>0$.
Then we have
\begin{equation}
\sum_{z\in N_P}h_z^{2(s+t)}\|Q_{\Lambda}Tv\|_{s,\omega(z)}^2
\lesssim
\sum_{z\in N_P}\|Q_{\Lambda}Tv\|_{-t,\omega(z)}^2
\leq
\|Q_{\Lambda}Tv\|_{-t}^2
\lesssim
\|Tv\|_{-t}^2
\lesssim
\|v\|_{t}^2,
\end{equation}
where we have used in succession a standard inverse estimate, the super-additivity
of the Sobolev norms \eqref{e:sobolev-add},
the stability of $Q_\Lambda$ in $H^{-t}$,
and the boundedness of $T$.

For the case $t>0$, a standard inverse estimate gives
\begin{equation}
\begin{split}
\sum_{z\in N_P}h_z^{2(s+t)}\|Q_{\Lambda}Tv\|_{s,\omega(z)}^2
&\lesssim
\sum_{z\in N_P}h_z^{2t}\|Q_{\Lambda}Tv\|_{\omega(z)}^2
\leq
\sum_{z\in N_P}\|h_P^{t}Q_{\Lambda}Tv\|_{\omega(z)}^2\\
&\lesssim
\|h_P^{t}Q_{\Lambda}Tv\|^2.
\end{split}
\end{equation}
At this point we employ the inverse estimate \eqref{e:wav-inverse},
to get
\begin{equation}
\|h_P^{t}Q_{\Lambda}Tv\|
\lesssim
\|Q_{\Lambda}Tv\|_{-t}
\lesssim
\|Tv\|_{-t}
\lesssim
\|v\|_{t},
\end{equation}
concluding the proof.
\end{proof}

What remains now is to bound $(I-Q_{\Lambda})Tv$, which consists of
only high frequency wavelets compared to what is in $Q_{\Lambda}Tv\in
S_\Lambda$.
To this end, for $\lambda\in\Lambda^c:=\nabla\setminus\Lambda$, let us define $\ell_\lambda$ by
\begin{equation}
2^{-\ell_\lambda} = \max\{h_\tau:\tau\in P,\,\tau\cap\supb_\lambda\neq\varnothing\},
\end{equation}
so that in light of the norm equivalence \eqref{e:bnormeq}, we can write
\begin{equation}\label{e:basic-wav-norm}
\begin{split}
\sum_{z\in N_P}h_z^{2(s+t)}\|(I-Q_{\Lambda})Tv\|_{s,\omega(z)}^2
&\lesssim
\sum_{z\in N_P}h_z^{2(s+t)}\sum_{\{\lambda\in\Lambda^c:\supb_\lambda\cap\,\omega(z)\neq\varnothing\}}2^{2|\lambda|(s+t)}|(Tv)_{\lambda}|^2\\
&\lesssim
\sum_{z\in N_P}\sum_{\{\lambda\in\Lambda^c:\supb_\lambda\cap\,\omega(z)\neq\varnothing\}}2^{-2\ell_\lambda(s+t)} 2^{2|\lambda|(s+t)}|(Tv)_{\lambda}|^2\\
&\lesssim
\sum_{\lambda\in\Lambda^c}2^{2(|\lambda|-\ell_\lambda)(s+t)}|(Tv)_{\lambda}|^2,
\end{split}
\end{equation}
where $(Tv)_{\lambda}=\langle Tv,\psi_{\lambda}\rangle$ is the
coordinate of $Tv$ with respect to $\tilde\psi_\lambda$, and in the
last step we have taken into account the fact that each
$\supb_\lambda$ intersects with only a uniformly bounded
number of stars $\omega(z)$.
Hence our aim is to bound the last expression in
\eqref{e:basic-wav-norm} by $\|(v_\lambda)_\lambda\|_{\ell^2}^2$ for
$v\in S_P^d$, where $v_\lambda=\langle v,\tilde\psi_\lambda\rangle$.

Before dealing with nonlocality of $T$, let us focus on the local
properties.
To this end, with $S_{P,z}=\{v\in S^d_P:v=0\textrm{ outside
}\omega^2(z)\}$,
define $Q_z$ to be the $L^2$-orthogonal projector onto $S_{P,z}$ in
case of piecewise constants,
and otherwise to be the quasi-interpolation operator onto
$S_{P,z}$, as discussed around \eqref{e:direct-quasi}.
Here $\omega^k(z)=\omega(\omega^{k-1}(z))$ for $k\geq2$.

\begin{lemma}[Local interactions]\label{l:local}
Let $s\geq0$ and let $t+s<\sigma$.
Then we have
\begin{equation}
\sum_{z\in N_P}h_z^{2(s+t)}\|(I-Q_{\Lambda})TQ_z v\|_{s,\omega(z)}^2
\lesssim
\|v\|_t^2
\qquad\textrm{for}\quad
v\in S^d_P.
\end{equation}
\end{lemma}

\begin{proof}
Replacing $v$ with $Q_zv$ in \eqref{e:basic-wav-norm}, we get
\begin{equation}
\sum_{z\in N_P}h_z^{2(s+t)}\|(I-Q_{\Lambda})TQ_z v\|_{s,\omega(z)}^2
\lesssim
\sum_{\lambda\in\Lambda^c}2^{2(|\lambda|-\ell_\lambda)(s+t)}|(TQ_z v)_{\lambda}|^2.
\end{equation}
In order to bound this, we view it as the (squared) $\ell^2(\Lambda^c)$-norm of a
sequence,
and we test the sequence against a general sequence $(w_\lambda)_\lambda\in \ell^2(\Lambda^c)$.
As a preparation to this, for $w\in\spn\{\psi_\lambda:|\lambda|=\ell\}$, we have
\begin{equation}\label{e:local-pf-1}
\begin{split}
|\langle w,TQ_z v\rangle|
&\lesssim \|w\|_{t-\sigma} \|TQ_z v\|_{-t+\sigma}
\lesssim 2^{-\sigma\ell} \|w\|_{t} \|Q_z v\|_{t+\sigma}\\
&\lesssim 2^{-\sigma\ell+(t+\sigma)\ell_z} \|w\|_{t} \|Q_z v\|
\lesssim 2^{-\sigma\ell+(t+\sigma)\ell_z} \|w\|_{t} \|v\|_{\omega^3(z)},
\end{split}
\end{equation}
where $\ell_z$ is defined by $2^{-\ell_z}=h_z$.
Then assuming that $t\leq0$, for $w=\sum_\lambda w_\lambda\psi_\lambda$ with $(w_\lambda)_\lambda\in \ell^2(\Lambda^c)$, we infer
\begin{multline}
\sum_{\lambda\in\Lambda^c}2^{(|\lambda|-\ell_\lambda)(s+t)}(TQ_z v)_{\lambda}w_\lambda
\lesssim
\sum_{z\in N_P}\sum_{\ell\geq\ell_z+c}
2^{(\ell-\ell_z)(s+t)}
\left|\left\langle
\sum_{\{|\lambda|=\ell:\supb_\lambda\cap\,\omega(z)\neq\varnothing\}}
w_\lambda\psi_\lambda,
TQ_z v\right\rangle\right|\\
\lesssim
\sum_{z\in N_P}\sum_{\ell\geq\ell_z+c}
2^{(\ell-\ell_z)(s+t-\sigma)+\ell_zt}
\left(\sum_{\{|\lambda|=\ell:\supb_\lambda\cap\,\omega(z)\neq\varnothing\}}
|w_\lambda|^2\right)^{\frac12}
\|v\|_{\omega^3(z)}\\
\lesssim
\sum_{z\in N_P}
2^{\ell_zt}
\left(\sum_{\{\lambda\in\Lambda^c:\supb_\lambda\cap\,\omega(z)\neq\varnothing\}}
|w_\lambda|^2\right)^{\frac12}
\|v\|_{\omega^3(z)}
\lesssim
\|w\|_{t}
\|h_P^{-t}v\|
\lesssim
\|w\|_{t}
\|v\|_{t},
\end{multline}
where $c\in\R$ is a constant that depends on $\delta$.
This establishes the lemma.

The case $t>0$ is proven similarly, by using
\begin{equation}
|\langle w,TQ_z v\rangle|
\lesssim 2^{-\sigma\ell+\sigma\ell_z} \|w\|_{t} \|v\|_{t,\omega^3(z)},
\end{equation}
instead of \eqref{e:local-pf-1}.
\end{proof}

To deal with nonlocality of $T$, 
we will employ the Schur test, which we recall here.

\begin{lemma}[Schur test]\label{l:schur}
Let $\Lambda$ and $M$ be countable sets,
and let $\ell^2(\Lambda,w)$ be the weighted $\ell^2$-space with the norm
\begin{equation}
\|v\|_{\ell^2(\Lambda,w)} = \|\{v_\lambda w_\lambda\}_{\lambda\in\Lambda}\|_{\ell^2(\Lambda)},
\end{equation}
where $\{w_\lambda\}_{\lambda\in\Lambda}$ is a given positive sequence
of weights.
For a matrix with entries $T_{\lambda\mu}$, $\lambda\in\Lambda$,
$\mu\in M$,
we consider its boundedness as a linear operator
$T:\ell^2(M)\to\ell^2(\Lambda,w)$.
Suppose that there exist two positive sequences
$\{\theta_\lambda\}_{\lambda\in\Lambda}$ and $\{\omega_\mu\}_{\mu\in
  M}$,
and two numbers $\alpha,\beta\in\R$ such that
\begin{equation}
\sum_{\mu\in M}|T_{\lambda\mu}|\omega_\mu\leq\alpha^2\theta_\lambda,
\qquad
\sum_{\lambda\in \Lambda}|T_{\lambda\mu}|w_\lambda^2\theta_\lambda\leq\beta^2\omega_\mu.
\end{equation}
Then we have
\begin{equation}
\|Tv\|_{\ell^2(\Lambda,w)} \leq \alpha\beta \|v\|_{\ell^2(M)},
\qquad v\in \ell^2(M).
\end{equation}
\end{lemma}

With the Riesz bases $\Psi\subset H^t(\Omega)$ and $\tilde\Psi\subset
H^{-t}(\Omega)$ at hand, the operator $T:H^t(\Omega)\to
H^{-t}(\Omega)$ can be thought of as the bi-infinite matrix $T:\ell^2(\nabla)\to
\ell^2(\nabla)$ with the elements $T_{\lambda\mu}=\langle \psi_{\lambda}, T \psi_{\mu} \rangle$.
The following estimate on these elements, established in \cite{Stev04},
will be crucial:
\begin{equation}
|T_{\lambda\mu}|
=|\langle \psi_{\lambda}, T \psi_{\mu} \rangle|
\lesssim 
\left(\frac{2^{-||\lambda|-|\mu||/2}}{\delta(\lambda,\mu)}\right)^{n+2t+2p}
\end{equation}
where
\begin{equation}
\delta(\lambda,\mu) = 2^{\min\{|\lambda|,|\mu|\}}\dist(\supb_\lambda,\supb_\mu).
\end{equation}
See also \cite{Schn98} and \cite{DS99} for earlier derivations for less general cases.

Let us fix a constant $\eps>0$, whose value is to be chosen later.
The following result shows that the far-field terms behave rather well.
Note that the condition $p+t>0$ is immaterial to us since we can choose $p$ at will.

\begin{lemma}[Far-field interactions]\label{l:far-field}
Define the operator $F:H^t(\Omega)\to H^{-t}(\Omega)$ with matrix elements
$F_{\lambda\mu} = T_{\lambda\mu}$
for $\lambda\in\Lambda^c$ and $\mu\in\nabla$ satisfying
$\dist(\supb_\lambda,\supb_\mu)\geq\eps\max\{2^{-\ell_\lambda},2^{-|\mu|}\}$, 
and $F_{\lambda\mu}=0$ otherwise.
Let $s+t>0$ and $p+t>0$.
Then we have
\begin{equation}
\sum_{\lambda\in\Lambda^c}2^{2(|\lambda|-\ell_\lambda)(s+t)}|(Fv)_{\lambda}|^2
\lesssim\|v\|_{t}^2
\qquad\textrm{for}\quad
v\in H^t(\Omega).
\end{equation}
\end{lemma}

\begin{proof}
We apply the Schur test (Lemma \ref{l:schur})
with $\theta_\lambda=2^{-|\lambda|n/2-(|\lambda|-\ell_\lambda)(d+t)}$, $\omega_\mu=2^{-|\mu|n/2}$, 
and $w_\lambda=2^{(|\lambda|-\ell_\lambda)(s+t)}$.
First we shall bound
\begin{equation}
\sum_{\mu\in\nabla}|F_{\lambda\mu}|2^{-|\mu|n/2}
\lesssim
\sum_{m\in\N} 2^{-mn/2} \sum_{\{|\mu|=m,\,F_{\lambda\mu}\neq0\}} 
\left(\frac{2^{-||\lambda|-m|/2}}{\delta(\lambda,\mu)}\right)^{n+2t+2p},
\end{equation}
by a multiple of $\theta_\lambda$.
In the inner sum, we must sum over all $\mu$ such that
$\delta(\lambda,\mu)\gtrsim2^{|\lambda|-\ell_\lambda}$
when $|\mu|=m\geq|\lambda|$,
and such that $\delta(\lambda,\mu)\gtrsim2^{\max\{0,m-\ell_\lambda\}}$ when $|\mu|=m\leq|\lambda|$.
By locality of the wavelets and the Lipschitz property $\Omega$,
for $\lambda\in\nabla$, $m\in\N$ and $\beta>0$, we have
\begin{equation}
\sum_{\{\mu:|\mu|=m,\,\delta(\lambda,\mu)\geq R\}}
\delta(\lambda,\mu)^{-(n+\beta)} \lesssim R^{-\beta} 2^{n\max\{0,m-|\lambda|\}},
\end{equation}
which appears, e.g., in \cite{Stev04}.
This gives
\begin{equation}
\sum_{\{\mu:|\mu|=m,\,\delta(\lambda,\mu)
  \gtrsim2^{|\lambda|-\ell_\lambda}\}}
\delta(\lambda,\mu)^{-(n+2t+2p)}
\lesssim
2^{-(|\lambda|-\ell_\lambda)(2p+2t)+(m-|\lambda|)n/2},
\end{equation}
for $m\geq|\lambda|$, and
\begin{equation}
\sum_{\{\mu:|\mu|=m,\,\delta(\lambda,\mu)
  \gtrsim2^{\max\{0,m-\ell_\lambda\}}\}}
\delta(\lambda,\mu)^{-(n+2t+2p)}
\lesssim
2^{-(|\lambda|-\ell_\lambda)(p+t)+(m-|\lambda|)n/2},
\end{equation}
for $m\leq|\lambda|$.
By using the preceding estimates, we conclude
\begin{equation}
\begin{split}
\sum_{\mu\in\nabla}|F_{\lambda\mu}|2^{-|\mu|n/2}
&\lesssim
\sum_{m\leq|\lambda|} 2^{-(|\lambda|-\ell_\lambda)(p+t)-(|\lambda|-m)(p+t)-|\lambda|n/2}\\
&\quad+ \sum_{m>|\lambda|}
2^{-(|\lambda|-\ell_\lambda)(2p+2t)-(m-|\lambda|)(p+t)-|\lambda|n/2}\\
&\lesssim
2^{-|\lambda|n/2-(|\lambda|-\ell_\lambda)(p+t)}.
\end{split}
\end{equation}

Now we shall bound
\begin{equation}\label{e:ff2}
\sum_{\lambda\in\Lambda^c} w_\lambda^2\theta_\lambda |F_{\lambda\mu}|
\lesssim
\sum_{\ell} \sum_{\{|\lambda|=\ell,\,F_{\lambda\mu}\neq0\}} 
2^{(\ell-\ell_\lambda)(2s+t-p)}
2^{-\ell n/2}
\left(\frac{2^{-|\ell-|\mu||/2}}{\delta(\lambda,\mu)}\right)^{n+2t+2p},
\end{equation}
by a multiple of $\omega_\mu$.
By construction, for $\lambda\in\Lambda^c$ and $\mu\in\nabla$ with
$F_{\lambda\mu}\neq0$,
we have $2^{-\min\{\ell_\lambda,|\mu|\}}
\lesssim\dist(\supb_\lambda,\supb_\mu)$,
implying that
\begin{equation}
2^{-\ell_\lambda}\lesssim\dist(\supb_\lambda,\supb_\mu),
\qquad\textrm{and}\qquad
\dist(\supb_\lambda,\supb_\mu) \gtrsim 2^{-\min\{|\lambda|,|\mu|\}}.
\end{equation}
In particular, the second estimate tells us that
$\delta(\lambda,\mu)\gtrsim1$ in the (inner) sum in \eqref{e:ff2}.
So for $\ell\leq|\mu|$, we estimate the inner sum as
\begin{equation}
\begin{split}
\sum_{|\lambda|=\ell} 
2^{(\ell-\ell_\lambda)(2s+t-p)}
&2^{-\ell n/2}
|F_{\lambda\mu}|
\leq
\sum_{|\lambda|=\ell} 
2^{2(\ell-\ell_\lambda)(s+t)}
2^{-\ell n/2}
|F_{\lambda\mu}|\\
&\lesssim
\sum_{\{|\lambda|=\ell,\,\delta(\lambda,\mu)\gtrsim1\}} 
\delta(\lambda,\mu)^{2(s+t)}
2^{-\ell n/2}
\left(\frac{2^{-(|\mu|-\ell)/2}}{\delta(\lambda,\mu)}\right)^{n+2t+2p}\\
&\lesssim
2^{-\ell n/2}
2^{-(|\mu|-\ell)(n/2+t+p)}
=
2^{-(|\mu|-\ell)(t+p)}
2^{-|\mu| n/2},
\end{split}
\end{equation}
where we have used $\ell\geq\ell_\lambda$ and $p+t>0$ in the first
inequality, and $s+t>0$ in the second.
Assuming that $2s+t-p\leq0$, for $\ell\geq|\mu|$, we estimate
\begin{equation}
\begin{split}
\sum_{|\lambda|=\ell} 
2^{(\ell-\ell_\lambda)(2s+t-p)}
2^{-\ell n/2}
|F_{\lambda\mu}|
&\leq
\sum_{|\lambda|=\ell} 
2^{-\ell n/2}
|F_{\lambda\mu}|\\
&\lesssim
2^{-\ell n/2}
2^{-(\ell-|\mu|)(n/2+t+p)}
2^{(\ell-|\mu|)n}\\
&=
2^{-(\ell-|\mu|)(p+t)}
2^{-|\mu| n/2}.
\end{split}
\end{equation}
Now if $2s+t-p>0$, for $\ell\geq|\mu|$, we have
\begin{equation}
\begin{split}
\sum_{|\lambda|=\ell} 
2^{(\ell-\ell_\lambda)(2s+t-p)}
2^{-\ell n/2}
|F_{\lambda\mu}|
&\lesssim
\sum_{|\lambda|=\ell} 
2^{(\ell-|\mu|)(2s+t-p)}\delta(\lambda,\mu)^{2s+t-p}
2^{-\ell n/2}
|F_{\lambda\mu}|\\
&\lesssim
2^{(\ell-|\mu|)(2s+t-p)} 
2^{-\ell n/2}
2^{-(\ell-|\mu|)(n/2+t+p)}
2^{(\ell-|\mu|)n}\\
&=
2^{-(\ell-|\mu|)(2p-2s)}
2^{-|\mu| n/2}.
\end{split}
\end{equation}
From the geometric decay of the preceding estimates,
we infer
\begin{equation}
\sum_{\lambda\in\Lambda^c}
2^{(|\lambda|-\ell_\lambda)(2s+t-p)}
2^{-|\lambda| n/2}
|F_{\lambda\mu}|
\lesssim
2^{-|\mu| n/2},
\end{equation}
which completes the proof.
\end{proof}

For the remaining terms, i.e., for the near-field terms,
we employ the simple estimate
\begin{equation}\label{e:near-est-naive}
\langle \psi_{\lambda}, T \psi_{\mu} \rangle
\lesssim 
\|\psi_{\lambda}\|_{t-\sigma} \|T\psi_{\mu}\|_{-t+\sigma}
\lesssim 
\|\psi_{\lambda}\|_{t-\sigma} \|\psi_{\mu}\|_{t+\sigma}
\lesssim 
2^{-|\lambda|\sigma+|\mu|\sigma}.
\end{equation}
We will see that this estimate gives sub-optimal results, that in general require
the manifold $\Omega$ to be smoother than Lipschitz.
There exist sharper estimates in the literature,
cf. \cite{Stev04,DHS06},
that exploit the piecewise smooth nature of the wavelets.
However, the author was not able to make use of them to get better results.
The best attempts so far by using the estimates from \cite{DHS06} resulted in
logarithmic divergences,
and the estimates from \cite{Stev04} in general need $\Omega$ to be smoother
than Lipschitz, thus do not seem to give improvements in this regard.

\begin{lemma}[Near-field interactions]\label{l:near-field}
Define the operator $N:H^{t}(\Omega)\to H^{-t}(\Omega)$ with matrix elements
$N_{\lambda\mu} = T_{\lambda\mu}$
for $\lambda\in\Lambda^c$ and $\mu\in\nabla$ satisfying
$\dist(\supb_\lambda,\supb_\mu)\leq2^{-|\mu|}$ and $\eps2^{|\mu|}\leq2^{\ell_\lambda}$, 
and $N_{\lambda\mu}=0$ otherwise.
Let $2\sigma>n$.
Then we have
\begin{equation}
\sum_{\lambda\in\Lambda^c}2^{2(|\lambda|-\ell_\lambda)(s+t)}|(Nv)_{\lambda}|^2
\lesssim
\|v\|_t^2
\qquad\textrm{for}\quad
v\in S^d_P.
\end{equation}
\end{lemma}

\begin{proof}
We apply the Schur test (Lemma \ref{l:schur})
with $\theta_\lambda= 2^{\varrho\ell_\lambda-\sigma|\lambda|}$, $\omega_\mu=2^{(\varrho-\sigma)|\mu|}$, 
and $w_\lambda=2^{(|\lambda|-\ell_\lambda)(s+t)}$, where
$\varrho=2 \sigma-n>0$.
We have
\begin{equation}
\sum_{\mu\in\nabla}|N_{\lambda\mu}| 2^{(\varrho-\sigma)|\mu|}
\lesssim
\sum_{m\leq\ell_\lambda} 2^{-\sigma|\lambda|+\varrho m}
\lesssim
2^{\varrho\ell_\lambda - \sigma|\lambda|},
\end{equation}
and
\begin{equation}
\begin{split}
\sum_{\lambda\in\Lambda^c} w_\lambda^2\theta_\lambda |N_{\lambda\mu}|
&\lesssim
\sum_{\{\tau\in P:\dist(\tau,\supb_\mu)\leq2^{-|\mu|}\}} \sum_{\ell\geq\ell_\tau} 
2^{2(\ell-\ell_\tau)(s+t)}
2^{\varrho \ell_\tau - \sigma\ell}
2^{-\sigma (\ell-|\mu|)} 
2^{n(\ell-\ell_\tau)}\\
&\lesssim
\sum_{\{\tau\in P:\dist(\tau,\supb_\mu)\leq2^{-|\mu|}\}}
2^{\varrho \ell_\tau - \sigma\ell_\tau}
2^{-\sigma (\ell_\tau-|\mu|)}\\
&=
\sum_{\{\tau\in P:\dist(\tau,\supb_\mu)\leq2^{-|\mu|}\}}
2^{\sigma|\mu| - n\ell_\tau}
\lesssim
\sum_{\{\tau\in P:\dist(\tau,\supb_\mu)\leq2^{-|\mu|}\}}
2^{\sigma|\mu|}\vol(\tau)\\
&\lesssim
2^{(\sigma-n)|\mu|}
=
2^{(\varrho-\sigma)|\mu|},
\end{split}
\end{equation}
which establishes the proof.
\end{proof}

We end this section by assembling the promised proof.

\begin{proof}[Proof of Theorem \ref{t:inverse-ineq}]
Let $v\in S^d_P$ be as in the theorem,
and consider the decomposition
\begin{equation}
Tv=Q_{\Lambda}Tv + (I-Q_{\Lambda})Tv.
\end{equation}
The first term on the right hand side is the low frequency part, treated in Lemma \ref{l:low-freq}.
As discussed in \eqref{e:basic-wav-norm}, the second term is bounded as
\begin{equation}
\sum_{z\in N_P}h_z^{2(s+t)}\|(I-Q_{\Lambda})Tv\|_{s,\omega(z)}^2
\lesssim
\sum_{\lambda\in\Lambda^c}2^{2(|\lambda|-\ell_\lambda)(s+t)}|(Tv)_{\lambda}|^2.
\end{equation}
Comparing the definitions of the near- and far-field interactions, we see that 
the combination of Lemma \ref{l:far-field} and Lemma \ref{l:near-field} takes care
of all the contributions, {\em except} those coming from the components $v_\mu$ with 
$\dist(\supb_\lambda,\supb_\mu)\leq\eps2^{-\ell_\lambda}$ and $2^{-|\mu|}\leq\eps2^{-\ell_\lambda}$.
But by choosing $\eps>0$ small enough, we can absorb the contributions of those missing components
into the local interactions, which is accounted in Lemma \ref{l:local}.
Note that there is no issue with the dependence on $\eps$ of the hidden constants in Lemma \ref{l:far-field} and Lemma \ref{l:near-field},
since the choice of $\eps>0$ is done once and for all, uniformly in $v\in S^d_P$ for $P\in\adm(P_0)$.
\end{proof}

\section[Conclusion and References]{Concluding remarks}
\label{s:conclude}

In this paper, we proved geometric error reduction for
three kinds of adaptive boundary element methods for positive, negative, as well as zero
order operator equations, without relying on a saturation-type assumption.
In fact, several types of saturation assumptions follow from our work as a corollary.
Moreover, bounds on the convergence rates are obtained that are in
a certain sense optimal.

We established several new global- and local discrete bounds for a number of residual-type error estimators for positive, negative, as well zero order
boundary integral equations, including the estimators from
\citet{CMS01}, \citet{CMPS04}, and \cite{Faer00,Faer02}.
Some of our bounds contain oscillation terms,
that give useful estimates on how far the current mesh is from saturation.
In order to handle the oscillation terms, which turned out to be not strtaightforward,
we introduced an inverse-type inequality involving
boundary integral operators and locally refined meshes.
Our proof of the inequality in general requires the underlying surface $\Gamma$ to be
$C^{1,1}$ or smoother, but for open surfaces it allows the boundary of
$\Gamma$ to be Lipschitz.
So in general, polyhedral surfaces are ruled out, which is very likely an artifact
of the proof, since 
in \cite{FKMP11a}, the inequality is proven for a model negative order operator on polyhedral surfaces.

The current work gives rise to its fair share of open problems, and re-emphasizes some existing ones.
The following is an attempt at identifying the most pressing of them.
\begin{itemize}
\item
to prove the inverse-type inequality in \S\ref{s:inverse} for
polyhedral surfaces, for $t\geq0$.
This is established for Symm's integral operator (with $t=-\frac12$) in \cite{FKMP11a}.
\item
to characterize the approximation classes associated to the proposed
adaptive BEMs.
Some progress on this question has been made in \cite{AFFKP12}.
\item
to generalize the proofs to higher order boundary element methods.
\item
to extend the analysis to transmission problems, and adaptive FEM-BEM coupling.
\item
convergence rate for adaptive
BEMs based on non-residual type error estimators.
\item
complexity analysis, i.e., the problem of quadrature and linear algebra solvers.
In particular, one would like to know how accurate the residual should be computed in the error estimators.
\end{itemize}

\section*{Acknowledgements}

I would like to thank Dirk Praetorius for carefully reading an earlier version of this manuscript, and for making several important comments.
I thank the anonymous referees for their reviews and suggestions.
I also thank Michael Renardy, Nilima Nigam, and
Elias Pipping over at \href{http://mathoverflow.net}{\texttt{mathoverflow.net}} for pointing out the reference \cite{Kim07},
and Doyoon Kim for making his paper available to me.
This work is supported by an NSERC Discovery Grant and an FQRNT Nouveaux
Chercheurs Grant.



\bibliographystyle{plainnat}
\bibliography{../bib/timur}

\end{document}